\documentclass[a4paper,10pt,english,french]{smfart}
\usepackage[frenchb,english]{babel}
\usepackage[T1]{fontenc}
\usepackage{lmodern}
\usepackage{graphicx}
\usepackage{tabularx}
\usepackage{amsmath,amsthm,amssymb,amsfonts}
\usepackage[a4paper,left=4cm,right=4cm,top=4cm,bottom=4cm]{geometry}
\numberwithin{equation}{section}
\title[Affine congruences and rational points on a certain cubic surface]{Affine congruences and rational points on a certain cubic surface}
\author{Pierre Le Boudec}
\subjclass{$11$D$45$, $14$G$05$}
\address{Institute for Advanced Study \\ School of Mathematics \\ Einstein Drive \\
Simonyi Hall $-$ \text{Office $111$} \\ Princeton, NJ $08540$ \\ USA}
\email{pleboudec@ias.edu}

\begin{document}

\makeatletter
\def\imod#1{\allowbreak\mkern10mu({\operator@font mod}\,\,#1)}
\makeatother

\newtheorem{lemma}{Lemma}
\newtheorem{theorem}{Theorem}
\newtheorem{corollaire}{Corollaire}
\newtheorem{proposition}{Proposition}

\newcommand{\vol}{\operatorname{vol}}
\newcommand{\D}{\mathrm{d}}
\newcommand{\rank}{\operatorname{rank}}
\newcommand{\Pic}{\operatorname{Pic}}
\newcommand{\Gal}{\operatorname{Gal}}
\newcommand{\meas}{\operatorname{meas}}
\newcommand{\Spec}{\operatorname{Spec}}
\newcommand{\eff}{\operatorname{eff}}
\newcommand{\rad}{\operatorname{rad}}
\newcommand{\sq}{\operatorname{sq}}

\begin{abstract}
We establish estimates for the number of solutions of certain affine congruences. These estimates are then used to prove Manin's conjecture for a cubic surface split over $\mathbb{Q}$ and whose singularity type is $\mathbf{D}_4$. This improves on a result of Browning and answers a problem posed by Tschinkel.
\end{abstract}

\maketitle

\tableofcontents

\section{Introduction}

The aim of this paper is to study the asymptotic behaviour of the number of rational points of bounded height on the cubic surface $V \subset \mathbb{P}^3$ defined over $\mathbb{Q}$ by
\begin{eqnarray*}
x_0(x_1+x_2+x_3)^2 - x_1x_2x_3 & = & 0 \textrm{.}
\end{eqnarray*}
Manin's conjecture \cite{MR974910} and the refinements concerning the value of the constant due to Peyre \cite{MR1340296} and to Batyrev and Tschinkel \cite{MR1679843} describe precisely what should be the solution of this problem.

The variety $V$ has a unique singularity at the point $(1:0:0:0)$, which is of type $\mathbf{D}_4$. In addition, it contains precisely six lines which are defined by $x_0 = x_i = 0$ and $x_1+x_2+x_3 = x_i =0$ for $i \in \{1,2,3\}$. On these six lines, the rational points accumulate, hiding the interesting behaviour of the number of rational points lying outside the lines. We thus let $U$ be the open subset formed by removing the six lines from $V$. We also let $H : \mathbb{P}^3(\mathbb{Q}) \to \mathbb{R}_{> 0}$ be the exponential height defined for a vector
$(x_0, x_1, x_2, x_3) \in \mathbb{Z}^{4}$ satisfying $\gcd(x_0, x_1, x_2, x_3) = 1$ by
\begin{eqnarray*}
H(x_0 : x_1 : x_2 : x_3) & = & \max \{ |x_0|, |x_1|, |x_2|, |x_3| \}\textrm{.}
\end{eqnarray*}
The quantity in which we are interested is then defined by
\begin{eqnarray*}
N_{U,H}(B) & = & \# \{x \in U(\mathbb{Q}), H(x) \leq B \} \textrm{.}
\end{eqnarray*}
In this specific context, Manin's conjecture states that
\begin{eqnarray*}
N_{U,H}(B) & = & c_{V,H} B \log(B)^6 (1+o(1)) \textrm{,}
\end{eqnarray*}
where $c_{V,H}$ is a constant which is expected to agree with Peyre's prediction. In a more general setting, the exponent of the logarithm is expected to be equal to the rank of the Picard group of the minimal desingularization of $V$ minus one. In comparison, the number $N_{\mathbb{P}^1,H}(B)$ of rational points of bounded height lying on a line satisfies $N_{\mathbb{P}^1,H}(B) = c_{\mathbb{P}^1,H} B^2 (1+o(1))$ where $c_{\mathbb{P}^1,H} > 0$.

Manin's conjecture for singular cubic surfaces has received an increasing amount of attention over the last years (see for instance \cite{MR2430199}, \cite{MR2332351} and \cite{2A2+A1}). The interested reader is invited to refer to the recent work of the author \cite[Section~$1$]{2A2+A1} for a comprehensive overview of what is currently known concerning singular cubic surfaces defined over $\mathbb{Q}$.

Any cubic surface in $\mathbb{P}^3$ defined over $\mathbb{C}$, which has only isolated singularities, and which is not a cone over an elliptic curve can only have $\mathbf{ADE}$ singularities (see \cite[Proposition $0.2$]{MR940430}). In Table $1$ below, we recall the classification over $\overline{\mathbb{Q}}$ of cubic surfaces with $\mathbf{ADE}$ singularities and we give the number of lines contained by the surfaces. Moreover, we indicate if Manin's conjecture is known for at least one example of surface of the specified singularity type by giving the corresponding reference. Note that the difficulty of proving Manin's conjecture increases as we go higher in Table $1$.

\vspace{10pt}
\begin{center}
\renewcommand{\arraystretch}{1.2}
\begin{tabular}{|c|c|c|}
\hline
Singularity type & Number of lines & Result \\
\hline
\hline
$\mathbf{A}_1$ & $21$ &  \\
\hline
$2 \mathbf{A}_1$ & $16$ &  \\
\hline
$\mathbf{A}_2$ & $15$ & \\
\hline
$3 \mathbf{A}_1$ & $12$ & \\
\hline
$\mathbf{A}_2 + \mathbf{A}_1$ & $11$ & \\
\hline
$\mathbf{A}_3$ & $10$ & \\
\hline
$4 \mathbf{A}_1$ & $9$ & \\
\hline
$2 \mathbf{A}_1 + \mathbf{A}_2$ & $8$ & \\
\hline
$\mathbf{A}_3 + \mathbf{A}_1$ & $7$ & \\
\hline
$2 \mathbf{A}_2$ & $7$ & \\
\hline
$\mathbf{A}_4$ & $6$ & \\
\hline
$\mathbf{D}_4$ & $6$ & [\textbf{This paper}] \\
\hline
$2 \mathbf{A}_1 + \mathbf{A}_3$ & $5$ & \\
\hline
$2 \mathbf{A}_2 + \mathbf{A}_1$ & $5$ & \cite{2A2+A1}  \\
\hline
$\mathbf{A}_4 + \mathbf{A}_1$ & $4$ &  \\
\hline
$\mathbf{A}_5$ & $3$ & \\
\hline
$\mathbf{D}_5$ & $3$ & \cite{MR2520769}  \\
\hline
$3 \mathbf{A}_2$ & $3$ & \cite{MR1620682}  \\
\hline
$\mathbf{A}_5 + \mathbf{A}_1$ & $2$ & \cite{A5+A1}  \\
\hline
$\mathbf{E}_6$ & $1$ & \cite{MR2332351}  \\
\hline
\end{tabular}
\end{center}
\vspace{5pt}
\begin{center}
\textsc{Table} $1$. Cubic surfaces with $\mathbf{ADE}$ singularities.
\end{center}
\vspace{10pt}

At the American Institute of Mathematics workshop \textit{Rational and integral points on higher-dimensional varieties} in $2002$, Tschinkel posed the problem of studying the quantity $N_{U,H}(B)$. Motivated by the work of Heath-Brown \cite{MR2075628} dealing with Cayley's cubic surface, Browning \cite{MR2250046} brought a first answer to this question by proving that
\begin{eqnarray*}
N_{U,H}(B) & \asymp & B \log(B)^6 \textrm{,}
\end{eqnarray*}
where $\asymp$ means that the ratio of these two quantities is between two positive constants. To do so, he made use of the universal torsor which had been calculated by Hassett and Tschinkel \cite{MR2029868} and which is an open subset of the affine hypersurface embedded in
$\mathbb{A}^{10} \simeq \Spec \left( \mathbb{Q}[\eta_1, \dots, \eta_{10}] \right)$ and defined by
\begin{eqnarray*}
\eta_2 \eta_5^2 \eta_8 + \eta_3 \eta_6^2 \eta_9 + \eta_4 \eta_7^2 \eta_{10} - \eta_1 \eta_2 \eta_3 \eta_4 \eta_5 \eta_6 \eta_7 & = & 0 \textrm{.}
\end{eqnarray*}
In this paper, we also make use of this auxiliary variety to establish Manin's conjecture for $V$.

Let us note here that it is the first time that Manin's conjecture is proved for a del Pezzo surface for which the corresponding universal torsor is defined by the vanishing of a sum of four monomials (universal torsors are usually defined by the vanishing of a sum of three monomials).

Universal torsors have originally been introduced by Colliot-Thélène and Sansuc in order to study the Hasse principle and weak approximation for rational varieties (see \cite{MR0414556}, \cite{MR605344} and \cite{MR899402}). These descent methods have turned out to be a very pertinent tool for counting problems. The parametrizations of rational points provided by universal torsors have been used in the context of Manin's conjecture for the first time by Peyre \cite{MR1679842} and Salberger \cite{MR1679841}.

It is a well-established heuristic fact that counting rational points on cubic surfaces becomes harder as the number $N$ of $(-2)$-curves on the minimal desingularizations decreases (which means as we go higher in Table $1$). As a consequence, our result can be seen as a new record since $V$ is the first example of cubic surface with $N=4$ for which Manin's conjecture is proved. Previously, Manin's conjecture was known for only two non-toric cubic surfaces with $N=6$ (see \cite{MR2332351} and \cite{A5+A1}) and two cubic surfaces with $N=5$ (see \cite{MR2520769} and \cite{2A2+A1}).

Since the parametrizations of the rational points resorting to universal torsors become extremely complicated as $N$ decreases, it seems to the author that establishing Manin's conjecture for a cubic surface with $1 \leq N \leq 3$, and even for another cubic surface with $N = 4$, is an extremely difficult problem. In particular, all such surfaces have universal torsors which are not hypersurfaces. Actually, it is not even clear if sharp upper bounds for $N_{U,H}(B)$ can be obtained for surfaces with $1 \leq N \leq 3$. As a reminder, the best result known for non-singular cubic surfaces (id est with $N=0$) is the upper bound
\begin{eqnarray*}
N_{U,H}(B) & \ll & B^{4/3 + \varepsilon} \textrm{,}
\end{eqnarray*}
for any fixed $\varepsilon > 0$, which holds if the surface contains three coplanar lines defined over $\mathbb{Q}$ (see \cite{MR1438113}).

To prove Manin's conjecture for $V$, we start by establishing estimates for the number of $(u,v) \in \mathbb{Z}^2$ lying in a prescribed region and satisfying the congruence
\begin{eqnarray}
\label{cong}
a_1 u + a_2 v & \equiv & b \imod{q} \textrm{,}
\end{eqnarray}
and the condition $\gcd(uv,q)=1$, where $a_1, a_2 \in \mathbb{Z}_{\neq 0}$, $ q \in \mathbb{Z}_{\geq 1}$ are such that $a_1 a_2$ is coprime to $q$ and $b \in \mathbb{Z}$ is divisible by each prime number dividing $q$. Then, the first step of the proof consists in summing over three variables viewing the torsor equation as an affine congruence to which these estimates are applied.

At this stage of the proof, a very interesting phenomenon has to be noticed. The error term showing up in these estimates gives birth to a new congruence where the coefficients $a_1$ and $a_2$ appear. However, it is not possible to give a good bound for this quantity for any fixed $a_1$ and $a_2$ coprime to $q$. As a consequence, this quantity has to be estimated on average over certain variables dividing $a_1$ and $a_2$. More precisely, this error term is non-trivially summed over two other variables whose squares respectively divide $a_1$ and $a_2$ using a result  due to Heath-Brown and coming from the geometry of numbers.

The step which makes appear this new congruence is definitely the key step of our proof (see lemma \ref{Ramanujan lemma}). Our method is believed to be quite new and will certainly be useful to deal with other diophantine problems. For instance, the methods of lemmas~\ref{Ramanujan lemma} and~\ref{error} are used in forthcoming work of la Bretèche and Browning \cite{Hasse-Chatelet}, in which they study in a quantitative way the failure of the Hasse principle for a certain family of Châtelet surfaces.

It is worth pointing out that it is very likely that our work can be adapted to yield a proof of Manin's conjecture for another cubic surface with a single singularity of type $\mathbf{D}_4$ but lying in the other isomorphism class over $\overline{\mathbb{Q}}$ (there are exactly two isomorphism classes of cubic surfaces with $\mathbf{D}_4$ singularity type over $\overline{\mathbb{Q}}$). This cubic surface is defined over $\mathbb{Q}$ by
\begin{eqnarray*}
x_0(x_1+x_2+x_3)^2 + x_1 (x_1 + x_2) & = & 0 \textrm{,}
\end{eqnarray*}
and the universal torsor corresponding to this problem is an open subset of the affine hypersurface embedded in
$\mathbb{A}^{10} \simeq \Spec \left( \mathbb{Q}[\eta_1, \dots, \eta_{10}] \right)$ and defined by
\begin{eqnarray*}
\eta_2 \eta_5^2 \eta_8 + \eta_3 \eta_6^2 \eta_9 + \eta_4 \eta_7^2 \eta_{10} & = & 0 \textrm{.}
\end{eqnarray*}
The study of the congruence \eqref{cong} in the particular case $b = 0$ is expected to solve the problem of proving Manin's conjecture for this surface in a similar fashion.

Our main result is the following.

\begin{theorem}
\label{Manin}
As $B$ tends to $+ \infty$, we have the estimate
\begin{eqnarray*}
N_{U,H}(B) & = & c_{V,H} B \log(B)^{6} \left( 1 + O \left( \frac1{\log(\log(B))^{1/6}} \right) \right) \textrm{,}
\end{eqnarray*}
where $c_{V,H}$ agrees with Peyre's prediction.
\end{theorem}

It has been checked that $V$ is not an equivariant compactification of $\mathbb{G}_m^2$ or $\mathbb{G}_a^2$ (see \cite[Proposition $13$]{D-hyper} and
\cite{DL-equi}). Furthermore, let
\begin{eqnarray*}
G_d & = & \mathbb{G}_a \rtimes_d \mathbb{G}_m \textrm{,}
\end{eqnarray*}
where $d \in \mathbb{Z}$ and where the action of $g \in \mathbb{G}_m$ on $x \in \mathbb{G}_a$ is given by $g \cdot x = g^d x$. It can be checked that if $V$ were an equivariant compactification of $G_d$ then the number of negative curves on its minimal desingularization would be less or equal to $8$ which is not the case since this number is equal to $10$. As a result, theorem \ref{Manin} does not follow from the general results concerning equivariant compactifications of algebraic groups \cite{MR1620682}, \cite{MR1906155} and \cite{Semi-direct}.

The next section is dedicated to the proofs of several preliminary results. The two following sections are devoted to the respective descriptions of the universal torsor and of Peyre's constant. Finally, in the remaining section we prove theorem \ref{Manin}.

Along the proof, $\varepsilon$ is an arbitrarily small positive number. As a convention, the implicit constants involved in the notations $O$ and $\ll$ are always allowed to depend on $\varepsilon$.

It is a great pleasure for the author to thank his supervisor Professor de la Bretèche for his advice during the preparation of this work. The author is also grateful to Professor Browning for his careful reading of an earlier version of the manuscript.

This work has benefited from the financial support of the ANR PEPR (Points Entiers Points Rationnels).

\section{Preliminaries}

\subsection{Affine congruences}

Let $a_1, a_2 \in \mathbb{Z}_{\neq 0}$ be two integers and set $\mathbf{a} = \left(a_1,a_2\right)$. Let also $ q \in \mathbb{Z}_{\geq 1}$ and $b \in \mathbb{Z}$. We assume that $a_1 a_2$ is coprime to $q$. Moreover, if we let $\rad(n)$ denote the radical of an integer $n \geq 1$, that is to say
\begin{eqnarray*}
\rad(n) & = & \prod_{p \mid n} p \textrm{,}
\end{eqnarray*}
then we also assume that
\begin{eqnarray}
\label{rad}
\rad(q) & \mid & b \textrm{.}
\end{eqnarray}
Let $\mathcal{I}$ and $\mathcal{J}$ be two ranges. We introduce the quantities
\begin{eqnarray}
\label{N}
\ \ \ \ \ N(\mathcal{I},\mathcal{J};q,\mathbf{a},b) & = & \# \left\{ (u,v) \in \mathcal{I} \times \mathcal{J} \cap \mathbb{Z}^2,
\begin{array}{l}
a_1 u + a_2 v \equiv b \imod{q} \\
\gcd(uv,q) = 1
\end{array}
\right\} \textrm{,}
\end{eqnarray}
and
\begin{eqnarray}
\label{Nast}
N^{\ast}(\mathcal{I},\mathcal{J}; q) & = & \frac1{\varphi(q)} \# \left\{(u,v) \in \mathcal{I} \times \mathcal{J} \cap \mathbb{Z}^2, \gcd(uv,q) = 1 \right\} \textrm{.}
\end{eqnarray}

It is immediate to check that one of the two conditions among $\gcd(u,q)=1$ and $\gcd(v,q)=1$ can be omitted in the definition of $N(\mathcal{I},\mathcal{J};q,\mathbf{a},b)$. Indeed, if we omit the condition $\gcd(u,q)=1$ then the conditions $\gcd(a_2,q)=1$ and $\gcd(v,q)=1$ together imply that we have $\gcd(a_1u-b,q)=1$. Thanks to the conditions \eqref{rad} and $\gcd(a_1,q)=1$, this latter condition is seen to be equivalent to $\gcd(u,q)=1$.

Note that $N^{\ast}(\mathcal{I},\mathcal{J}; q)$ is the average of $N(\mathcal{I},\mathcal{J};q,\mathbf{a},b)$ over $a_1$ or $a_2$ coprime to $q$. In lemma \ref{Ramanujan lemma}, we show how we can approximate $N(\mathcal{I},\mathcal{J};q,\mathbf{a},b)$ by $N^{\ast}(\mathcal{I},\mathcal{J}; q)$. We start by studying some exponential sums which will naturally appear in the proof of lemma \ref{Ramanujan lemma}. For $q \in \mathbb{Z}_{\geq 1}$, we let $e_q$ be the function defined by $e_q\left(x\right) = e^{2 i \pi x /q}$ and we set, for $r, s \in \mathbb{Z}$,
\begin{eqnarray*}
S_q(r,s,\mathbf{a},b) & = & \sum_{\substack{\alpha, \beta = 1 \\ \gcd\left(\alpha \beta, q\right) = 1 \\ a_1 \alpha + a_2 \beta \equiv b \imod{q}}}^q
e_q(r \alpha + s \beta)  \textrm{.}
\end{eqnarray*}
Furthermore, we need to introduce the classical Ramanujan sum. For $q \in \mathbb{Z}_{\geq 1}$ and $n \in \mathbb{Z}$, we set
\begin{eqnarray*}
c_q\left(n\right) & = & \sum_{\substack{\alpha = 1 \\ \gcd(\alpha, q) = 1}}^{q} e_{q}( n \alpha)  \textrm{,}
\end{eqnarray*}
and we recall that
\begin{eqnarray}
\label{Ramanujan}
c_q(n) & = & \sum_{d | \gcd(q,n)} \mu \left( \frac{q}{d} \right) d \textrm{.}
\end{eqnarray}

\begin{lemma}
\label{exponential lemma}
For any $r, s \in \mathbb{Z}$, we have
\begin{eqnarray*}
S_q(r,s,\mathbf{a},b) & = & e_q \left( r a_1^{-1} b \right) c_q(a_1 s - a_2 r) \textrm{,}
\end{eqnarray*}
and symmetrically,
\begin{eqnarray*}
S_q(r,s,\mathbf{a},b) & = & e_q \left( s a_2^{-1} b \right) c_q(a_2 r - a_1 s) \textrm{,}
\end{eqnarray*}
where $a_1^{-1}$ and $a_2^{-1}$ denote respectively the inverses of $a_1$ and $a_2$ modulo $q$.

As a result, we have $S_q(q,s,\mathbf{a},b) = c_q(s)$ and $S_q(r,q,\mathbf{a},b) = c_q(r)$ and thus these two quantities are independent of $\mathbf{a}$ and $b$.
\end{lemma}

\begin{proof}
The symmetry given by the map $(r,s,a_1,a_2) \mapsto (s,r,a_2,a_1)$ implies that we only need to prove one of the two equalities. Let us prove the second one for instance. In a similar way as we can omit the condition $\gcd(v,q)=1$ in the definition of $N(\mathcal{I},\mathcal{J};q,\mathbf{a},b)$, we can also omit the condition $\gcd\left( \beta, q\right) = 1$ in the definition of $S_{q}(r,s,\mathbf{a},b)$. Therefore, we get
\begin{eqnarray*}
S_q(r,s,\mathbf{a},b) & = & \sum_{\substack{\alpha = 1 \\ \gcd\left(\alpha, q\right) = 1}}^q e_q \left(r \alpha \right)
\sum_{\substack{\beta = 1 \\ a_1 \alpha + a_2 \beta \equiv b \imod{q}}}^q e_q(s \beta) \\
& = & \sum_{\substack{\alpha = 1 \\ \gcd\left(\alpha, q\right) = 1}}^q e_q \left(r \alpha \right)
e_q \left( s \left( a_2^{-1} b - a_2^{-1} a_1 \alpha \right) \right) \\
& = & e_q \left( s a_2^{-1} b \right) \sum_{\substack{\alpha = 1 \\ \gcd\left(\alpha, q\right) = 1}}^q e_q \left( \left( r - a_2^{-1} a_1 s \right) \alpha \right) \\
& = & e_q \left( s a_2^{-1} b \right) c_q \left( r - a_2^{-1} a_1 s \right) \\
& = & e_q \left( s a_2^{-1} b \right) c_q \left( a_2 r - a_1 s \right) \textrm{,}
\end{eqnarray*}
as wished.
\end{proof}

From now on, for $\lambda > 0$, we define the arithmetic function $\sigma_{- \lambda}$ by
\begin{eqnarray*}
\sigma_{- \lambda}(n) & = & \sum_{k|n} k^{- \lambda} \textrm{.}
\end{eqnarray*}
We now prove the following lemma.

\begin{lemma}
\label{Ramanujan lemma}
We have the estimate
\begin{eqnarray*}
N(\mathcal{I}, \mathcal{J};q,\mathbf{a},b) - N^{\ast}(\mathcal{I}, \mathcal{J};q) & \ll & E(q,\mathbf{a})\textrm{,}
\end{eqnarray*}
where $E(q,\mathbf{a}) = E_0(q,\mathbf{a}) + E_1(q)$ and
\begin{eqnarray*}
E_0(q,\mathbf{a}) & = & \sum_{d|q} \left|  \mu \left( \frac{q}{d} \right)  \right| d \sum_{\substack{0 < |r|, |s| \leq q/2 \\ a_1 s - a_2 r \equiv 0 \imod{d}}} |r|^{-1} |s|^{-1} \textrm{,}
\end{eqnarray*}
and
\begin{eqnarray*}
E_1(q) & = & \left( \frac{q}{\varphi(q)} \right)^3 \log(q)^2 \textrm{.}
\end{eqnarray*}
\end{lemma}

\begin{proof}
We detect the congruence using sums of exponentials, we get
\begin{eqnarray*}
N(\mathcal{I},\mathcal{J};q,\mathbf{a},b) & = & \sum_{\substack{\alpha,\beta = 1 \\ \gcd(\alpha \beta,q) = 1 \\
a_1 \alpha + a_2 \beta \equiv b \imod{q}}}^q
\# \{ (u,v) \in \mathcal{I} \times \mathcal{J} \cap \mathbb{Z}^2, q| \alpha - u, \beta - v \} \\
& = & \! \! \! \! \! \! \sum_{\substack{\alpha,\beta = 1 \\ \gcd(\alpha \beta,q) = 1 \\ a_1 \alpha + a_2 \beta \equiv b \imod{q}}}^q \! \! \! \! \! \frac1{q^2}
\left( \sum_{u \in \mathcal{I}} \sum_{r=1}^q e_q(r \alpha - r u) \right)
\left( \sum_{v \in \mathcal{J}} \sum_{s=1}^q e_q(s \beta - s v) \right) \\
& = & \frac1{q^2} \sum_{r,s = 1}^q S_q(r,s,\mathbf{a},b)  F_q(r,s) \textrm{,}
\end{eqnarray*}
where
\begin{eqnarray*}
F_q(r,s) & = & \left( \sum_{u \in \mathcal{I}} e_q(- r u) \right) \left( \sum_{v \in \mathcal{J}} e_q(- s v) \right)
\textrm{.}
\end{eqnarray*}
Using lemma \ref{exponential lemma}, we get
\begin{eqnarray*}
N(\mathcal{I},\mathcal{J};q,\mathbf{a},b) & = &\frac1{q^2} \sum_{r,s = 1}^q e_q \left( r a_1^{-1} b \right) c_q(a_1 s - a_2 r) F_q(r,s) \textrm{.}
\end{eqnarray*}
Let $||x||$ denote the distance from $x$ to the set of integers. If $r,s \neq q$, $F_q(r,s)$ is a product of two geometric sums and we therefore have
\begin{eqnarray*}
F_q(r,s) & \ll & \left| \left| \frac{r}{q} \right| \right| ^{-1} \left| \left| \frac{s}{q} \right| \right| ^{-1} \textrm{.}
\end{eqnarray*}
Let $N(\mathcal{I}, \mathcal{J};q)$ be the sum of the terms corresponding to $r=q$ or $s=q$. As stated in lemma \ref{exponential lemma},
$N(\mathcal{I}, \mathcal{J};q)$ is independent of $a_1$, $a_2$ and $b$. Using the equality \eqref{Ramanujan}, we get
\begin{eqnarray*}
N(\mathcal{I},\mathcal{J};q,\mathbf{a},b) - N(\mathcal{I},\mathcal{J};q) & = & \frac1{q^2} \sum_{r,s = 1}^{q-1} e_q \left( r a_1^{-1} b \right) c_q(a_1 s - a_2 r) F_q(r,s) \\
& \ll & \frac1{q^2} \sum_{d|q}  \left|  \mu \left( \frac{q}{d} \right)  \right| d \sum_{\substack{r,s = 1 \\ a_1 s - a_2 r \equiv 0 \imod{d}}}^{q-1}
\left| \left| \frac{r}{q} \right| \right| ^{-1} \left| \left| \frac{s}{q} \right| \right| ^{-1} \\
& \ll & \frac1{q^2} \sum_{d|q}  \left|  \mu \left( \frac{q}{d} \right)  \right| d \sum_{\substack{0 < |r|, |s| \leq q/2 \\ a_1 s - a_2 r \equiv 0 \imod{d}}} \frac{q}{|r|} \frac{q}{|s|} \textrm{.}
\end{eqnarray*}
Recall that the right-hand side is equal to $E_0(q,\mathbf{a})$. We have thus obtained
\begin{eqnarray}
\label{estimate 1}
N(\mathcal{I},\mathcal{J};q,\mathbf{a},b) - N(\mathcal{I},\mathcal{J};q) & \ll & E_0(q,\mathbf{a}) \textrm{.}
\end{eqnarray}
Since $N(\mathcal{I}, \mathcal{J};q)$ is independent of $a_2$ and since $N^{\ast}(\mathcal{I},\mathcal{J};q)$ is the average of
$N(\mathcal{I},\mathcal{J};q,\mathbf{a},b)$ over $a_2$ coprime to $q$, averaging this estimate over $a_2$ coprime to $q$ shows that
\begin{eqnarray*}
N^{\ast}(\mathcal{I},\mathcal{J};q) - N(\mathcal{I},\mathcal{J};q) & \ll & E_1'(q) \textrm{,}
\end{eqnarray*}
where
\begin{eqnarray*}
E_1'(q)  & = & \frac1{\varphi(q)} \sum_{d \mid q} d \sum_{0 < |r|, |s| \leq q/2} |r|^{-1} |s|^{-1}
\sum_{\substack{a_2 = 1 \\ \gcd(a_2,q)=1 \\ a_1s-a_2r \equiv 0 \imod{d}}}^q 1 \\
& \ll & \frac1{\varphi(q)} \sum_{d \mid q} d \sum_{0 < |r|, |s| \leq q/2} \gcd(r,s,d) |r|^{-1} |s|^{-1}  \\
& \ll & \frac1{\varphi(q)} \sum_{d \mid q} d \sum_{d' \mid d} d' \sum_{\substack{0 < |r|, |s| \leq q/2 \\ d' \mid r, d' \mid s}} |r|^{-1} |s|^{-1} \\
& \ll & \frac1{\varphi(q)} \log(q)^2 \sum_{d \mid q} d \sigma_{-1}(d) \textrm{.}
\end{eqnarray*}
Furthermore, we can check that the right-hand side is bounded by $E_1(q)$. Thus
\begin{eqnarray}
\label{estimate 2}
N^{\ast}(\mathcal{I},\mathcal{J};q) - N(\mathcal{I},\mathcal{J};q) & \ll & E_1(q) \textrm{,}
\end{eqnarray}
and therefore, combining the estimates \eqref{estimate 1} and \eqref{estimate 2}, we obtain
\begin{eqnarray*}
N(\mathcal{I},\mathcal{J};q,\mathbf{a},b) - N^{\ast}(\mathcal{I},\mathcal{J};q) & \ll & E(q,\mathbf{a}) \textrm{,}
\end{eqnarray*}
which completes the proof.
\end{proof}

Note that an immediate consequence of lemma \ref{Ramanujan lemma} is the bound
\begin{eqnarray}
\label{average bound}
N(\mathcal{I}, \mathcal{J};q,\mathbf{a},b) & \ll & \frac1{\varphi(q)} \# \left( \mathcal{I} \times \mathcal{J} \cap \mathbb{Z}^2 \right) + E(q,\mathbf{a}) \textrm{.}
\end{eqnarray}

We now introduce a certain domain $\mathcal{S} \subset \mathbb{R}^2$ where the couple $(u,v)$ is restricted to lie. Let $X, T, A_1, A_2 \geq 1$. We let
$\mathcal{S} = \mathcal{S}(X, T, A_1, A_2)$ be the set of $(x,y) \in \mathbb{R}^2$ such that
\begin{eqnarray}
\label{A}
A_1 |x| A_2 |y| |A_1 x + A_2 y - T| & \leq & T^2 X \textrm{,} \\
\label{B}
|A_1 x + A_2 y - T| & \leq & X \textrm{,} \\
\label{C}
A_1 |x| & \leq & X \textrm{,} \\
\label{D}
A_2 |y| & \leq & X \textrm{.}
\end{eqnarray}
Note that the last three conditions imply that we also have
\begin{eqnarray*}
T & \leq & 3 X \textrm{.}
\end{eqnarray*}
Finally, we set
\begin{eqnarray*}
D(\mathcal{S};q,\mathbf{a},b) & = & \# \left\{ (u,v) \in \mathcal{S} \cap \mathbb{Z}_{\neq 0}^2,
\begin{array}{l}
a_1 u + a_2 v \equiv b \imod{q} \\
\gcd(uv,q) = 1
\end{array}
\right\} \textrm{,}
\end{eqnarray*}
and
\begin{eqnarray*}
D^{\ast}(\mathcal{S};q) & = & \frac1{\varphi(q)} \# \left\{ (u,v) \in \mathcal{S} \cap \mathbb{Z}_{\neq 0}^2, \gcd(uv,q) = 1 \right\} \textrm{.}
\end{eqnarray*}

We now aim to prove the following lemma.

\begin{lemma}
\label{lemma affine}
Let $L \geq 1$. We have the estimate
\begin{eqnarray*}
D(\mathcal{S};q,\mathbf{a},b) - D^{\ast}(\mathcal{S};q) & \ll & \frac1{L} \frac{X^3}{T A_1 A_2 \varphi(q)} + L^4 \log(2X)^2 E(q,\mathbf{a}) \textrm{.}
\end{eqnarray*}
\end{lemma}

The proof of lemma \ref{lemma affine} requires a technical result similar to \cite[Lemma $4$]{MR2853047}. The analysis of the proof of \cite[Lemma $4$]{MR2853047} immediately shows that the following lemma holds.

\begin{lemma}
\label{square}
Let $0 < \nu \leq 1$ and $M_0 \in \mathbb{R}_{>0}$. Let $Y  \in \mathbb{R}_{>0}$ and $A, Y' \in \mathbb{R}$ be such that $0 < Y - Y' \ll \nu M_0^2$ and set
$M = \max \left( |A|, Y^{1/2} \right)$. Let $\mathcal{R} \subset \mathbb{R}$ be the set of real numbers $y$ subject to
\begin{eqnarray}
\label{double inequality}
& & Y' < \left| y^2 + 2 A y \right| \leq Y \textrm{.}
\end{eqnarray}
We have the bound
\begin{eqnarray*}
\# \left( \mathcal{R} \cap \mathbb{Z} \right) & \ll &  \nu \frac{M_0^2}{M} + \nu^{1/2} M_0 + 1 \textrm{.}
\end{eqnarray*}
In particular, if $M_0 \geq M$ then we have the bound
\begin{eqnarray*}
\# \left( \mathcal{R} \cap \mathbb{Z} \right) & \ll &  \nu^{1/2} \frac{M_0^2}{M} + 1  \textrm{.}
\end{eqnarray*}
\end{lemma}

Let us now prove lemma \ref{lemma affine}.

\begin{proof}
If $\mathcal{S} \cap \mathbb{Z}_{\neq 0}^2 = \emptyset$ then the result is obvious. We therefore assume from now on that
$\mathcal{S} \cap \mathbb{Z}_{\neq 0}^2 \neq \emptyset$. We let $0 < \delta, \delta' \leq 1$ be two parameters to be selected in due course and we set
$\zeta = 1 + \delta$ and $\zeta' = 1 + \delta'$. In addition, we let $U$ and $V$ be variables running respectively over the sets
$\{ \pm \zeta^n, n \in \mathbb{Z}_{\geq -1} \}$ and $\{ \pm \zeta'^n, n \in \mathbb{Z}_{\geq -1} \}$. We define $\mathcal{I} = ]U,\zeta U]$ if $U > 0$ and $\mathcal{I} = [\zeta U,U[$ if $U < 0$ and the range $\mathcal{J}$ is defined the same way using the variable $V$ and the parameter $\zeta'$. We have
\begin{eqnarray*}
D(\mathcal{S};q,\mathbf{a},b) -  \sum_{\mathcal{I} \times \mathcal{J} \cap \mathbb{Z}^2 \subset \mathcal{S}}
N(\mathcal{I},\mathcal{J};q,\mathbf{a},b) & \ll &
\sum_{\substack{\mathcal{I} \times \mathcal{J} \cap \mathbb{Z}^2 \nsubseteq \mathcal{S} \\ \mathcal{I} \times \mathcal{J} \cap \mathbb{Z}^2 \nsubseteq \mathbb{R}^2 \setminus \mathcal{S}}}
N(\mathcal{I},\mathcal{J};q,\mathbf{a},b) \textrm{.}
\end{eqnarray*}
We define the quantity
\begin{eqnarray*}
D(\mathcal{S};q) & = & \sum_{\mathcal{I} \times \mathcal{J} \cap \mathbb{Z}^2 \subset \mathcal{S}}
N^{\ast}(\mathcal{I},\mathcal{J};q) \textrm{.}
\end{eqnarray*}
We note here that since $N^{\ast}(\mathcal{I},\mathcal{J};q)$ is independent of $a_1$, $a_2$ and $b$, $D(\mathcal{S};q)$ is also independent of $a_1$, $a_2$ and $b$. Moreover, we have
\begin{eqnarray*}
\sum_{\mathcal{I} \times \mathcal{J} \cap \mathbb{Z}^2 \subset \mathcal{S}}
N(\mathcal{I},\mathcal{J};q,\mathbf{a},b) - D(\mathcal{S};q) & \ll & \frac{\log(2X)^2}{\delta \delta'} E(q,\mathbf{a}) \textrm{,}
\end{eqnarray*}
where we have used lemma \ref{Ramanujan lemma} and noted that the number of rectangles $\mathcal{I} \times \mathcal{J}$ such that
$\mathcal{I} \times \mathcal{J} \cap \mathbb{Z}^2 \subset \mathcal{S}$ is less than
\begin{eqnarray*}
4 \left( 1 + \frac{\log(X)}{\log(\zeta)} \right) \left( 1 + \frac{\log(X)}{\log(\zeta')} \right) & \ll & \frac{\log(2X)^2}{\delta \delta'} \textrm{,}
\end{eqnarray*}
since $\delta, \delta' \leq 1$. We have proved that
\begin{eqnarray*}
D(\mathcal{S};q,\mathbf{a},b) - D(\mathcal{S};q) & \ll &
\sum_{\substack{\mathcal{I} \times \mathcal{J} \cap \mathbb{Z}^2 \nsubseteq \mathcal{S} \\
\mathcal{I} \times \mathcal{J} \cap \mathbb{Z}^2 \nsubseteq \mathbb{R}^2 \setminus \mathcal{S}}}
N(\mathcal{I},\mathcal{J};q,\mathbf{a},b) + \frac{\log(2X)^2}{\delta \delta'} E(q,\mathbf{a}) \textrm{.}
\end{eqnarray*}
Using the bound \eqref{average bound} for $N(\mathcal{I},\mathcal{J};q,\mathbf{a},b)$, we conclude that
\begin{eqnarray*}
D(\mathcal{S};q,\mathbf{a},b) - D(\mathcal{S};q) \! & \ll & \! \frac1{\varphi(q)} \! \! \! \! \!
\sum_{\substack{\mathcal{I} \times \mathcal{J} \cap \mathbb{Z}^2 \nsubseteq \mathcal{S} \\ \mathcal{I} \times \mathcal{J} \cap \mathbb{Z}^2 \nsubseteq \mathbb{R}^2 \setminus \mathcal{S}}} \! \! \! \! \!
\# \left( \mathcal{I} \times \mathcal{J} \cap \mathbb{Z}^2 \right) + \frac{\log(2X)^2}{\delta \delta'} E(q,\mathbf{a}) \textrm{,}
\end{eqnarray*}
since the number of rectangles $\mathcal{I} \times \mathcal{J}$ satisfying
$\mathcal{I} \times \mathcal{J} \cap \mathbb{Z}^2 \nsubseteq \mathcal{S}$ and
$\mathcal{I} \times \mathcal{J} \cap \mathbb{Z}^2 \nsubseteq \mathbb{R}^2 \setminus \mathcal{S}$ is also
$\ll \log(2X)^2 \delta^{-1} \delta'^{-1}$. The sum of the right-hand side is over all the rectangles $\mathcal{I} \times \mathcal{J}$ for which we have $(\zeta^{s_1} U, \zeta'^{s_2} V) \in \mathcal{S} \cap \mathbb{Z}^2$ and
$(\zeta^{t_1} U, \zeta'^{t_2} V) \in \mathbb{Z}^2 \setminus \mathcal{S}$ for some $(s_1,s_2) \in ]0,1]^2$ and $(t_1,t_2) \in ]0,1]^2$. This means that one of the inequalities defining $\mathcal{S}$ is not satisfied by
$(\zeta^{t_1} U, \zeta'^{t_2} V)$ and we need to estimate the contribution coming from each condition among \eqref{A}, \eqref{B}, \eqref{C} and \eqref{D}. Note that we always have the conditions
\begin{eqnarray}
\label{C'}
A_1 |U| & \leq & X \textrm{,} \\
\label{D'}
A_2 |V| & \leq & X \textrm{.}
\end{eqnarray}
In what follows, we could sometimes write strict inequalities instead of non-strict ones but this would not change anything in our reasoning. Let us first deal with the condition \eqref{A}. For the rectangles
$\mathcal{I} \times \mathcal{J}$ described above, for some $(s_1,s_2) \in ]0,1]^2$ and $(t_1,t_2) \in ]0,1]^2$, we have
\begin{eqnarray}
\label{condition1}
\zeta^{s_1} \zeta'^{s_2} A_1 |U| A_2 |V| \left| \zeta^{s_1} A_1 U + \zeta'^{s_2} A_2 V - T \right| & \leq & T^2 X \textrm{,} \\
\label{condition2}
\zeta^{t_1}  \zeta'^{t_2} A_1 |U| A_2 |V| \left| \zeta^{t_1} A_1 U + \zeta'^{t_2} A_2 V - T \right| & > & T^2 X \textrm{.}
\end{eqnarray}
These two conditions imply respectively
\begin{eqnarray*}
\left| A_1 U + A_2 V - T \right| & \leq & \frac{T^2 X}{A_1|U| A_2 |V|} + \delta A_1|U| + \delta' A_2|V| \textrm{,}
\end{eqnarray*}
and
\begin{eqnarray*}
\left| A_1 U + A_2 V - T \right| & > & \zeta^{-1}  \zeta'^{-1} \frac{T^2 X}{A_1|U| A_2 |V|} - \delta A_1|U| - \delta' A_2|V| \textrm{.}
\end{eqnarray*}
Setting $\Delta = \delta + \delta'$, we thus get
\begin{eqnarray}
\label{condition3}
& & \zeta^{-1}  \zeta'^{-1} \frac{T^2 X}{A_1|U| A_2 |V|} - \Delta X < \left| A_1 U + A_2 V - T \right| \leq \frac{T^2 X}{A_1|U| A_2 |V|} + \Delta X \textrm{.}
\end{eqnarray}
Going back to the variables $u$ and $v$, it is immediate to check that
\begin{eqnarray*}
\big| \left| A_1 u + A_2 v - T \right| - \left| A_1 U + A_2 V - T \right| \big| & \leq & \delta A_1|U| + \delta' A_2|V| \\
& \leq & \Delta X \textrm{.}
\end{eqnarray*}
Therefore, the inequality \eqref{condition3} gives
\begin{eqnarray*}
& & \zeta^{-1} \zeta'^{-1} \frac{T^2 X}{A_1|u| A_2 |v|} - 2 \Delta X < \left| A_1 u + A_2 v - T \right| \leq \zeta \zeta' \frac{T^2 X}{A_1|u| A_2 |v|} + 2 \Delta X \textrm{.}
\end{eqnarray*}
Finally, we obtain the condition
\begin{eqnarray}
\label{condition u,v}
& & \zeta^{-1} \zeta'^{-1} \frac{T^2 X}{A_1^2 A_2 |v|} - 4 \Delta \frac{X^2}{A_1^2} < |u| \left| u + \frac{A_2}{A_1} v - \frac{T}{A_1} \right|
\leq \zeta \zeta' \frac{T^2 X}{A_1^2 A_2 |v|} + 4 \Delta \frac{X^2}{A_1^2} \textrm{.}
\end{eqnarray}
Since $T \leq 3 X$, we can apply the second estimate of lemma \ref{square} with
\begin{eqnarray*}
M_0 & = & \frac{X^{3/2}}{A_1 A_2^{1/2} |v|^{1/2}} \textrm{,}
\end{eqnarray*}
and $\nu = \Delta$. We see that the error we want to estimate is bounded by
\begin{eqnarray*}
\sum_{\substack{\eqref{C'}, \eqref{D'} \\ \eqref{condition3}}}
\# \left( \mathcal{I} \times \mathcal{J} \cap \mathbb{Z}^2 \right)
& \ll & \# \left\{ (u,v) \in \mathbb{Z}_{\neq 0}^2,
\begin{array}{l}
\eqref{condition u,v} \\
|u| \ll X/ A_1 \\
|v| \ll X/ A_2
\end{array}
\right\} \\
& \ll & \sum_{|v| \ll X / A_2} \left( \Delta^{1/2} \frac{X^{5/2}}{T A_1 A_2^{1/2} |v|^{1/2}}+ 1 \right) \\
& \ll & \Delta^{1/2} \frac{X^3}{TA_1A_2} + \frac{X}{A_2} \textrm{.}
\end{eqnarray*}
Using the symmetry between the variables $u$ and $v$, we see that we also have
\begin{eqnarray*}
\sum_{\substack{\eqref{C'}, \eqref{D'} \\ \eqref{condition3}}} \# \left( \mathcal{I} \times \mathcal{J} \cap \mathbb{Z}^2 \right) & \ll &
\Delta^{1/2}\frac{X^3}{TA_1A_2} + \frac{X}{A_1} \textrm{,}
\end{eqnarray*}
and thus
\begin{eqnarray*}
\sum_{\substack{\eqref{C'}, \eqref{D'} \\ \eqref{condition3}}} \# \left( \mathcal{I} \times \mathcal{J} \cap \mathbb{Z}^2 \right) & \ll &
\Delta^{1/2} \frac{X^3}{TA_1A_2} + \frac{X}{A_1^{1/2} A_2^{1/2}} \textrm{.}
\end{eqnarray*}
We now reason in a similar way to treat the cases of the other conditions. Let us estimate the contribution coming from the condition \eqref{B}. We see that the condition which plays the role of \eqref{condition3} in the previous case is here
\begin{eqnarray}
\label{condition4}
& & X - \Delta X < |A_1 U + A_2 V - T| \leq X + \Delta X \textrm{.}
\end{eqnarray}
Furthermore, going back to the variables $u$ and $v$, we  obtain
\begin{eqnarray}
\label{condition5}
& & X - 2 \Delta X < |A_1 u + A_2 v - T| \leq X + 2 \Delta X \textrm{.}
\end{eqnarray}
We therefore find that the error in this case is bounded by
\begin{eqnarray*}
\sum_{\substack{\eqref{C'}, \eqref{D'} \\ \eqref{condition4}}}
\# \left( \mathcal{I} \times \mathcal{J} \cap \mathbb{Z}^2 \right)
& \ll & \# \left\{ (u,v) \in \mathbb{Z}_{\neq 0}^2,
\begin{array}{l}
\eqref{condition5} \\
|u| \ll X / A_1 \\
|v| \ll X / A_2
\end{array}
\right\} \\
& \ll & \sum_{|v| \ll X / A_2} \left( \Delta \frac{X}{A_1} + 1 \right) \\
& \ll & \Delta \frac{X^2}{A_1A_2} + \frac{X}{A_2} \textrm{.}
\end{eqnarray*}
Once again, using the symmetry between the variables $u$ and $v$, we obtain
\begin{eqnarray*}
\sum_{\substack{\eqref{C'}, \eqref{D'} \\ \eqref{condition4}}}
\# \left( \mathcal{I} \times \mathcal{J} \cap \mathbb{Z}^2 \right)
& \ll & \Delta \frac{X^2}{A_1A_2} + \frac{X}{A_1^{1/2} A_2^{1/2}} \textrm{.}
\end{eqnarray*}
Finally, if $X / A_1 < 2$ then it is clear that we do not have to consider the case of the condition \eqref{C} and if $X / A_1 \geq 2$ then we are going to choose $\delta$ such that $X/A_1$ is an integer power of $\zeta$ and, as a result, we do not have to consider the case of this condition here either. The same reasoning holds for the choice of the parameter $\delta'$ depending on the size of the quantity $X/A_2$. As a consequence, we have obtained
\begin{eqnarray*}
D(\mathcal{S};q,\mathbf{a},b) - D(\mathcal{S};q) & \ll & \Delta^{1/2} \frac{X^3}{T A_1 A_2 \varphi(q)} + \frac{\log(2X)^2}{\delta \delta'} E(q,\mathbf{a})
+ \frac{X}{A_1^{1/2} A_2^{1/2} \varphi(q)} \textrm{.}
\end{eqnarray*}
Note that if $q=1$ then the result of lemma \ref{lemma affine} is clear and if $q > 1$ then the third term of the right hand-side is always dominated by one of the two others. Let $L \geq 1$. We can choose
\begin{eqnarray*}
\delta, \delta' & \asymp & \frac1{L^2} \textrm{,}
\end{eqnarray*}
such that $\zeta$ and $\zeta'$ are respectively integer powers of $X/A_1$ and $X/A_2$ if these quantities are greater or equal to $2$. These choices of $\delta$ and $\delta'$ give
\begin{eqnarray*}
D(\mathcal{S};q,\mathbf{a},b) - D(\mathcal{S};q) & \ll & \frac1{L} \frac{X^3}{T A_1 A_2 \varphi(q)} + L^4 \log(2X)^2 E(q,\mathbf{a}) \textrm{.}
\end{eqnarray*}
Since $D(\mathcal{S};q)$ does not depend on $a_2$ and since $D^{\ast}(\mathcal{S};q)$ is the average of $D(\mathcal{S};q,\mathbf{a},b)$ over $a_2$ coprime to $q$, averaging the latter estimate over $a_2$ coprime to $q$ yields
\begin{eqnarray*}
D^{\ast}(\mathcal{S};q) - D(\mathcal{S};q) & \ll &  \frac1{L} \frac{X^3}{T A_1 A_2 \varphi(q)} + L^4 \log(2X)^2 E_1(q)  \textrm{.}
\end{eqnarray*}
Putting these two estimates together completes the proof.
\end{proof}

Note that the application of lemma \ref{square} could have been achieved less crudely using the first estimate of this lemma instead. However, we will see that this does not matter much for our purpose. Indeed, the estimate for $D^{\ast}(\mathcal{S};q)$ given in lemma \ref{volume D} proves that the result of lemma
\ref{lemma affine} is interesting only if $T$ is not too small compared to $X$. Fortunately, in the setting of the proof of theorem \ref{Manin}, we will be able to restrict ourselves to a situation in which $T$ and $X$ have comparable orders of magnitude, as stated in lemma \ref{lemma affine final}.

Our next aim is to approximate the cardinality which appears in $D^{\ast}(\mathcal{S};q)$ by its corresponding two-dimensional volume. For this, we define the real-valued function
\begin{eqnarray}
\label{equation h}
h & : & (x,y,t) \mapsto \max \left\{ |x y| \left| x + y - t \right|, t^2|x|, t^2|y|,  t^2|x + y - t| \right\} \textrm{.}
\end{eqnarray}
It is immediate to check that
\begin{eqnarray}
\label{equality S}
\mathcal{S} & = & \left\{ (x,y) \in \mathbb{R}^2, h \left( \frac{A_1 x}{X^{1/3} T^{2/3}}, \frac{A_2 y}{X^{1/3} T^{2/3}}, \frac{T^{1/3}}{X^{1/3}} \right) \leq 1 \right\} \textrm{.}
\end{eqnarray}
We also introduce the real-valued functions
\begin{eqnarray}
\notag
g_1 & : & (y,t) \mapsto \int_{h(x,y,t) \leq 1} \D x \textrm{,} \\
\label{def g_2}
g_2 & : & t \mapsto \int g_1(y,t) \D y \textrm{.}
\end{eqnarray}

We now prove the following lemma.

\begin{lemma}
\label{bounds}
For $(y,t) \in \mathbb{R} \times \mathbb{R}_{>0}$, we have the bounds
\begin{eqnarray*}
g_1(y,t) & \ll & t^{-2} \textrm{,} \\
g_2(t) & \ll & 1 \textrm{.}
\end{eqnarray*}
\end{lemma}

\begin{proof}
The bound for $g_1$ is clear since $t^2|x| \leq 1$. To prove the bound for $g_2$, we use the elementary result \cite[Lemma $5$.$1$]{MR2520770}. We obtain
\begin{eqnarray*}
\int_{|x y| \left| x + y - t \right| \leq 1} \D x & \ll & \min \left\{ \frac1{|y|^{1/2}}, \frac1{|y||y-t|} \right\} \textrm{.}
\end{eqnarray*}
Therefore, we have
\begin{eqnarray*}
g_2(t) & \ll & \int_{|y| \leq 1} \frac{\D y}{|y|^{1/2}} + \int_{|y|, |y-t| \geq 1} \frac{\D y}{|y||y-t|} +
 \int_{|y| \geq 1, |y-t| \leq 1}  \frac{\D y}{|y|^{3/4}|y-t|^{1/2}} \textrm{.}
\end{eqnarray*}
The three terms of the right-hand side are easily seen to be bounded by an absolute constant, which completes the proof.
\end{proof}

We now prove that the following result holds.

\begin{lemma}
\label{volume D}
We have the estimate
\begin{eqnarray*}
D^{\ast}(\mathcal{S};q) - \frac{\varphi(q)}{q^2} \frac{X^{2/3} T^{4/3}}{A_1 A_2} g_2 \left( \frac{T^{1/3}}{X^{1/3}} \right) & \ll &
\frac{X^2}{A_1 A_2 q} \left( \frac{A_1^{1/2}}{X^{1/2}} + \frac{A_2^{1/2}}{X^{1/2}} \right) E_2(q) \textrm{,}
\end{eqnarray*}
where
\begin{eqnarray*}
E_2(q) & = & \frac{q}{\varphi(q)} \sigma_{-1/2}(q) \sigma_{-1}(q) \textrm{.}
\end{eqnarray*}
\end{lemma}

\begin{proof}
We start by removing the two coprimality conditions $\gcd(u,q)=1$ and $\gcd(v,q)=1$ using Möbius inversions. We get
\begin{eqnarray}
\label{estimate D}
D^{\ast}(\mathcal{S};q) & = & \frac1{\varphi(q)} \sum_{\ell_1 \mid q} \mu(\ell_1) \sum_{\ell_2 \mid q} \mu(\ell_2) C(\ell_1, \ell_2, \mathcal{S}) \textrm{,}
\end{eqnarray}
where
\begin{eqnarray*}
C(\ell_1, \ell_2, \mathcal{S}) & = & \# \left\{ \left( u', v' \right) \in \mathbb{Z}_{\neq 0}^2, \left( \ell_1 u', \ell_2 v' \right) \in \mathcal{S} \right\} \textrm{.}
\end{eqnarray*}
To count the number of $u'$ to be considered, we use the estimate
\begin{eqnarray*}
\# \left\{ n \in \mathbb{Z}_{\neq 0} \cap [t_1,t_2] \right\} & = & t_2 - t_1 + O \left( \max (|t_1|,|t_2|)^{1/2} \right) \textrm{.}
\end{eqnarray*}
We obtain
\begin{eqnarray*}
C(\ell_1, \ell_2, \mathcal{S}) & = & \sum_{\substack{v' \in \mathbb{Z}_{\neq 0} \\ A_2 \ell_2 |v'| \leq X}}
\left( \frac{X^{1/3} T^{2/3}}{A_1 \ell_1} g_1 \left( \frac{A_2 \ell_2 v'}{X^{1/3} T^{2/3}}, \frac{T^{1/3}}{X^{1/3}} \right)
+ O\left( \frac{X^{1/2}}{A_1^{1/2} \ell_1^{1/2}} \right) \right) \\
& = & \frac{X^{1/3} T^{2/3}}{A_1 \ell_1} \sum_{\substack{v' \in \mathbb{Z}_{\neq 0} \\ A_2 \ell_2 |v'| \leq X}}
g_1 \left( \frac{A_2 \ell_2 v'}{X^{1/3} T^{2/3}}, \frac{T^{1/3}}{X^{1/3}} \right) + O \left( \frac{X^{3/2}}{A_1^{1/2} \ell_1^{1/2} A_2 \ell_2} \right)
\textrm{.}
\end{eqnarray*}
The first bound of lemma \ref{bounds} implies that
\begin{eqnarray*}
\sup_{|y| \leq X^{2/3}/T^{2/3}} g_1 \left( y, \frac{T^{1/3}}{X^{1/3}} \right) & \ll & \frac{X^{2/3}}{T^{2/3}} \textrm{.}
\end{eqnarray*}
Since $g_1$ is easily seen to have a piecewise continuous derivative, this bound and an application of partial summation yield
\begin{eqnarray*}
\sum_{\substack{v' \in \mathbb{Z}_{\neq 0} \\ A_2 \ell_2 |v'| \leq X}} g_1 \left( \frac{A_2 \ell_2 v'}{X^{1/3} T^{2/3}}, \frac{T^{1/3}}{X^{1/3}} \right) & = &
\frac{X^{1/3} T^{2/3}}{A_2 \ell_2} g_2 \left( \frac{T^{1/3}}{X^{1/3}} \right) + O \left( \frac{X^{7/6}}{T^{2/3} A_2^{1/2} \ell_2^{1/2}} \right) \textrm{.}
\end{eqnarray*}
We have finally proved that
\begin{eqnarray*}
C(\ell_1, \ell_2, \mathcal{S}) & = & \frac1{\ell_1 \ell_2} \frac{X^{2/3} T^{4/3}}{A_1 A_2} g_2 \left( \frac{T^{1/3}}{X^{1/3}} \right)
+ O \left( \frac{X^{3/2}}{A_1 \ell_1 A_2^{1/2} \ell_2^{1/2}} + \frac{X^{3/2}}{A_1^{1/2} \ell_1^{1/2}  A_2\ell_2} \right)
\textrm{.}
\end{eqnarray*}
Putting together this estimate and the equality \eqref{estimate D} completes the proof.
\end{proof}

One of the immediate consequences of lemmas \ref{lemma affine} and \ref{volume D} is the following result, which corresponds exactly to the setting of the proof of theorem \ref{Manin}.

\begin{lemma}
\label{lemma affine final}
Let $L \geq 1$ and $\mathcal{L} \geq 1$. If
\begin{eqnarray*}
\frac{X}{\mathcal{L}} & \leq & T \textrm{,}
\end{eqnarray*}
then we have the estimate
\begin{eqnarray*}
D(\mathcal{S};q,\mathbf{a},b) - \frac{\varphi(q)}{q^2} \frac{X^{2/3} T^{4/3}}{A_1 A_2} g_2 \left( \frac{T^{1/3}}{X^{1/3}} \right) & \ll & E \textrm{,}
\end{eqnarray*}
where $E = E(X,T,A_1,A_2,L,\mathcal{L},q,\mathbf{a})$ is given by
\begin{eqnarray*}
E & = & L^4 \log(2X)^2 E(q,\mathbf{a}) + \frac{X^{2/3} T^{4/3}}{A_1 A_2 q} \mathcal{L}^{4/3} \left( \frac{\mathcal{L}}{L} + \frac{A_1^{1/2}}{X^{1/2}} + \frac{A_2^{1/2}}{X^{1/2}} \right) E_2(q) \textrm{.}
\end{eqnarray*}
\end{lemma}

\subsection{The error term}

We now turn to the investigation of the error term $E(q,\mathbf{a}')$ in the particular case where $\mathbf{a}' = (b_1 c_1^2, b_2 c_2^2)$ for $b_1, b_2, c_1, c_2 \in \mathbb{Z}_{\geq 1}$. We aim to give an upper bound for the sums of $E(q,\mathbf{a}')$ over $c_1$ and $c_2$ in some dyadic ranges. For this, we make use the following result which comes from the geometry of numbers and which is due to Heath-Brown (see \cite[Lemma $3$]{MR757475}). Note that this result had already been used by Browning in \cite{MR2250046} to prove that $N_{U,H}(B)$ has the expected order of magnitude.

\begin{lemma}
\label{geometry lemma}
Let $(v_1,v_2,v_3) \in \mathbb{Z}^3$ be a primitive vector and let $W_1, W_2, W_3 \geq 1$. The number of primitive vectors $(w_1,w_2,w_3) \in \mathbb{Z}^3$ satisfying the conditions $|w_i| \leq W_i$ for  $i = 1,2,3$ and the equation
\begin{eqnarray*}
v_1 w_1 + v_2 w_2 + v_3 w_3 & = & 0 \textrm{,}
\end{eqnarray*}
 is at most
\begin{eqnarray*}
12 \pi \frac{W_1 W_2 W_3}{ \max \left\{ |v_i|W_i \right\} } + 4 \textrm{,}
\end{eqnarray*}
where the maximum is taken over $i=1,2,3$.
\end{lemma}

From now on, we let $\tau$ be the usual divisor function. Recall the definitions of $E(q,\mathbf{a}')$ and $E_1(q)$ given in lemma \ref{Ramanujan lemma}. We now prove the following lemma.

\begin{lemma}
\label{error}
Let $C_1, C_2 \geq 1/2$. We have the bound
\begin{eqnarray*}
\sideset{}{^\ast} \sum_{\substack{C_i < c_i \leq 2C_i \\ \gcd(c_1,c_2)=1}}
E(q,\mathbf{a}') & \ll & \left( C_1 C_2 \tau(q)+ q \right) 2^{\omega(q)} E_1(q) \textrm{,}
\end{eqnarray*}
where the notation $\sum^{\ast}$ means that the summation is restricted to integers which are coprime to $q$ and where $i$ implicitly runs over the set $\{1,2\}$.
\end{lemma}

\begin{proof}
We have
\begin{eqnarray*}
\sideset{}{^\ast} \sum_{\substack{C_i < c_i \leq 2C_i \\ \gcd(c_1,c_2)=1}} E(q,\mathbf{a}') & \ll &
\sideset{}{^\ast} \sum_{\substack{C_i < c_i \leq 2C_i \\ \gcd(c_1,c_2)=1}} E_0(q,\mathbf{a}') + C_1 C_2 E_1(q) \textrm{.}
\end{eqnarray*}
The first term of the right-hand side is less than
\begin{eqnarray*}
& & \sum_{d|q} d \sum_{0 < |r|, |s| \leq q/2} |r|^{-1} |s|^{-1} \sum_{\substack{C_i < c_i \leq 2C_i \\ \gcd(c_1,c_2)=1 \\ b_1 c_1^2 s - b_2 c_2^2 r \equiv 0 \imod{d}}} 1 \textrm{.}
\end{eqnarray*}
Let us set $g = \gcd(r,s,d)$ and $s' = s/g$, $r' = r/g$ and $d'=d/g$. We have
\begin{eqnarray*}
\sideset{}{^\ast} \sum_{\substack{C_i < c_i \leq 2C_i \\ \gcd(c_1,c_2)=1 \\ b_1 c_1^2 s - b_2 c_2^2 r \equiv 0 \imod{d}}} 1 & = &
\sum_{\substack{1 \leq \rho \leq d \\ b_1 s \rho^2 - b_2 r \equiv 0 \imod{d}}} \ \ \ \ \
\sideset{}{^\ast} \sum_{\substack{C_i < c_i \leq 2C_i \\ \gcd(c_1,c_2)=1 \\ \rho c_2 \equiv c_1 \imod{d}}} 1 \\
& = & \sum_{\substack{1 \leq \rho \leq d \\ b_1 s' \rho^2 - b_2 r' \equiv 0 \imod{d'}}} \ \ \ \ \
\sideset{}{^\ast} \sum_{\substack{C_i < c_i \leq 2C_i \\ \gcd(c_1,c_2)=1 \\ \rho c_2 \equiv c_1 \imod{d}}} 1 \textrm{.}
\end{eqnarray*}
Using lemma \ref{geometry lemma}, we get
\begin{eqnarray*}
\sideset{}{^\ast} \sum_{\substack{C_i < c_i \leq 2C_i \\ \gcd(c_1,c_2)=1 \\ \rho c_2 \equiv c_1 \imod{d}}} 1 & \ll &
\frac{C_1 C_2}{d} + 1 \textrm{.}
\end{eqnarray*}
As a consequence, we have proved that
\begin{eqnarray*}
\sideset{}{^\ast} \sum_{\substack{C_i < c_i \leq 2C_i \\ \gcd(c_1,c_2)=1 \\ b_1 c_1^2 s - b_2 c_2^2 r \equiv 0 \imod{d}}} 1 & \ll &
\gcd(r,s,d) 2^{\omega(d)} \left( \frac{C_1 C_2}{d} + 1 \right) \textrm{.}
\end{eqnarray*}
Finally, we easily get
\begin{eqnarray*}
\sum_{d|q} d \sum_{0 < |r|, |s| \leq q/2} |r|^{-1} |s|^{-1} \frac{\gcd(r,s,d) 2^{\omega(d)}}{d} & \ll &
\sum_{d|q} 2^{\omega(d)} \sum_{e|d} e \sum_{\substack{0 < |r|, |s| \leq q/2 \\ e|r, e|s}} |r|^{-1} |s|^{-1} \\
& \ll & 2^{\omega(q)} \tau(q) \sigma_{-1}(q) \log(q)^2 \\
& \ll & 2^{\omega(q)} \tau(q) E_1(q) \textrm{,}
\end{eqnarray*}
and, as in the proof of lemma \ref{Ramanujan lemma}, we obtain
\begin{eqnarray*}
\sum_{d|q} d \sum_{0 < |r|, |s| \leq q/2} |r|^{-1} |s|^{-1} \gcd(r,s,d) 2^{\omega(d)} & \ll & q 2^{\omega(q)} E_1(q) \textrm{.}
\end{eqnarray*}
As a result, we have proved that
\begin{eqnarray*}
\sideset{}{^\ast} \sum_{\substack{C_i < c_i \leq 2C_i \\ \gcd(c_1,c_2)=1}} E_0(q,\mathbf{a}') & \ll & (C_1 C_2 \tau(q) + q) 2^{\omega(q)} E_1(q) \textrm{,}
\end{eqnarray*}
which completes the proof.
\end{proof}

\subsection{Arithmetic functions}

\label{arithmetic section}

We now introduce several arithmetic functions which will appear along the proof of theorem \ref{Manin}. We set

\begin{eqnarray}
\label{varphi}
\varphi^{\ast}(n) & = & \prod_{p|n} \left( 1 - \frac1{p} \right) \textrm{,} \\
\label{diamond}
\varphi^{\curlyvee}(n) & = & \prod_{p|n} \left( 1 - \frac1{p} \right)^{-2} \left( 1 + \frac{2}{p} \right)^{-1} \textrm{,}
\end{eqnarray}
and also, for $a, b \in \mathbb{Z}_{\geq 1}$,
\begin{eqnarray}
\label{checkmark}
\psi_a(n) & = & \prod_{\substack{p|n \\ p \nmid a}} \left( 1 - \frac1{p} \right)^2 \left( 1 - \frac1{p-1} \right) \textrm{,}
\end{eqnarray}
and
\begin{eqnarray*}
\psi_{a,b}(n) & = &
\begin{cases}
\psi_a(n) & \textrm{ if } \gcd(n,b) = 1 \textrm{,} \\
0 & \textrm{ otherwise.}
\end{cases}
\end{eqnarray*}

Following the straightforward reasonings of the proofs of \cite[Lemmas $5$, $6$]{MR2853047}, we easily obtain the following result.

\begin{lemma}
\label{arithmetic preliminary 1}
Let $0 < \gamma \leq 1$ be fixed. Let $0 \leq t_1 < t_2$ and set $I = [t_1,t_2]$. Let $g : \mathbb{R}_{> 0} \to \mathbb{R}$ be a function having a piecewise continuous derivative on $I$ whose sign changes at most $R_g(I)$ times on $I$. We have
\begin{eqnarray*}
\sum_{n \in I \cap \mathbb{Z}_{>0}} \psi_{a,b}(n) g(n) & = & \Upsilon \Psi(a,b) \int_I g(t) \D t + O_{\gamma} \left( \sigma_{- \gamma/2}(ab) t_2^{\gamma} M_I(g) \right) \textrm{,}
\end{eqnarray*}
where
\begin{eqnarray}
\label{Upsilon}
\Upsilon & = &  \prod_{p} \varphi^{\curlyvee}(p)^{-1} \textrm{,} \\
\notag
\Psi(a,b) & = & \varphi^{\ast}(b) \varphi^{\curlyvee}(ab) \textrm{,}
\end{eqnarray}
and
\begin{eqnarray*}
M_I(g) & = & (1 + R_g(I)) \sup_{t \in I \cap \mathbb{R}_{> 0}} |g(t)| \textrm{.}
\end{eqnarray*}
\end{lemma}

\section{The universal torsor}

\label{torsor section}

In this section we define a bijection between the set of rational points of bounded height on $U$ and a certain set of integral points on the hypersurface defined in the introduction. The universal torsor corresponding to our present problem has first been determined by Hassett and Tschinkel \cite{MR2029868} and then, it has been used by Browning in \cite{MR2250046} to prove the lower and upper bounds of the expected order of magnitude for $N_{U,H}(B)$. We employ the notation used by Derenthal in
\cite{D-hyper}. Let $\mathcal{T}(B)$ be the set of $(\eta_1, \dots, \eta_{10}) \in \mathbb{Z}_{>0}^7 \times \mathbb{Z}_{\neq 0}^3$ satisfying the equation
\begin{eqnarray}
\label{torsor}
\eta_2 \eta_5^2 \eta_8 + \eta_3 \eta_6^2 \eta_9 + \eta_4 \eta_7^2 \eta_{10} - \eta_1 \eta_2 \eta_3 \eta_4 \eta_5 \eta_6 \eta_7 & = & 0 \textrm{,}
\end{eqnarray}
the coprimality conditions
\begin{eqnarray}
\label{gcd1}
& & \gcd(\eta_{10}, \eta_1 \eta_2 \eta_3 \eta_4 \eta_5 \eta_6) = 1 \textrm{,} \\
\label{gcd2}
& & \gcd(\eta_9, \eta_1 \eta_2 \eta_3 \eta_4 \eta_5 \eta_7) = 1 \textrm{,} \\
\label{gcd3}
& & \gcd(\eta_8, \eta_1 \eta_2 \eta_3 \eta_4 \eta_6 \eta_7) = 1 \textrm{,} \\
\label{gcd4}
& & \gcd(\eta_1, \eta_5 \eta_6 \eta_7) = 1 \textrm{,} \\
\label{gcd5}
& & \gcd(\eta_2 \eta_5, \eta_3 \eta_4 \eta_6 \eta_7) = 1 \textrm{,} \\
\label{gcd6}
& & \gcd(\eta_3 \eta_6, \eta_4 \eta_7) = 1 \textrm{,}
\end{eqnarray}
and the height conditions
\begin{eqnarray}
\label{height1}
|\eta_8 \eta_9 \eta_{10}| & \leq & B \textrm{,} \\
\label{height2}
\eta_1^2 \eta_2^2 \eta_3 \eta_4 \eta_5^2 |\eta_8| & \leq & B \textrm{,} \\
\label{height3}
\eta_1^2 \eta_2 \eta_3^2 \eta_4 \eta_6^2 |\eta_9| & \leq & B \textrm{,} \\
\label{height4}
\eta_1^2 \eta_2 \eta_3 \eta_4^2 \eta_7^2 |\eta_{10}| & \leq & B \textrm{.}
\end{eqnarray}
We now prove the following lemma.

\begin{lemma}
\label{T}
We have the equality
\begin{eqnarray*}
N_{U,H}(B) & = & \# \mathcal{T}(B) \textrm{.}
\end{eqnarray*}
\end{lemma}

\begin{proof}
It is sufficient to show that the counting problem defined by the set $\mathcal{T}(B)$ is equivalent to the one described in the work of Browning \cite[Section $4$]{MR2250046}, which we call $\mathcal{T}'(B)$ and which is defined exactly as $\mathcal{T}(B)$ except that the condition  \eqref{gcd4} is replaced by the condition $|\mu(\eta_2 \eta_3 \eta_4)| = 1$.

For $i=2,3,4$, there is only one way to write $\eta_i = \eta_i' \eta_i''^2$ requiring that $\eta_i'$ is squarefree. Setting $\eta_{i+3}' = \eta_{i+3} \eta_i''$ and $\eta_1' = \eta_1 \eta_2'' \eta_3'' \eta_4''$, we claim that the translation between the two counting problems is achieved via the map
\begin{eqnarray*}
S & : & (\eta_1, \eta_2, \eta_3, \eta_4, \eta_5, \eta_6, \eta_7) \mapsto (\eta_1', \eta_2', \eta_3', \eta_4', \eta_5', \eta_6', \eta_7') \textrm{.}
\end{eqnarray*}
Indeed, the equation \eqref{torsor} and the height conditions \eqref{height1}, \eqref{height2}, \eqref{height3} and \eqref{height4} are invariant under $S$. Moreover, the coprimality conditions \eqref{gcd1}, \eqref{gcd2}, \eqref{gcd3}, \eqref{gcd5} and \eqref{gcd6}  are preserved under $S$, and the condition \eqref{gcd4} is replaced by the condition $|\mu(\eta_2' \eta_3' \eta_4')| = 1$, which completes the proof.
\end{proof}

\section{Calculation of Peyre's constant}

In \cite{MR1340296}, Peyre gives an interpretation for the constant $c_{V,H}$ appearing in the main term of $N_{U,H}(B)$ in theorem \ref{Manin}. In our specific case, we have
\begin{eqnarray*}
c_{V,H} & = & \alpha(\widetilde{V}) \beta(\widetilde{V}) \omega_H(\widetilde{V}) \textrm{,}
\end{eqnarray*}
where $\widetilde{V}$ denotes the minimal desingularization of $V$. The definitions of these three quantities are omitted (the reader should refer to \cite{MR1340296} or to \cite[Section  $4$]{2A2+A1} for some more details in an identical setting). Using the work of Derenthal, Joyce and Teitler
\cite[Theorem $1$.$3$]{MR2377367}, it is easy to compute the constant $\alpha(\widetilde{V})$. We find
\begin{eqnarray*}
\alpha(\widetilde{V}) & = & \frac1{120} \cdot \frac1{\# W(\mathbf{D}_4)} \\
& = & \frac1{23040} \textrm{,}
\end{eqnarray*}
where $W(\mathbf{D}_4)$ stands for the Weyl group associated to the Dynkin diagram of the singularity $\mathbf{D}_4$. Here, we have used
$\# W(\mathbf{D}_n) = 2^{n-1} n!$ for any $n \geq 4$. In addition, $\beta(\widetilde{V}) = 1$ since $V$ is split over $\mathbb{Q}$. Finally, $\omega_H(\widetilde{V})$ is given by
\begin{eqnarray*}
\omega_H(\widetilde{V}) & = & \omega_{\infty} \prod_p \left( 1 - \frac1{p} \right)^{7} \omega_p \textrm{,}
\end{eqnarray*}
where $\omega_{\infty}$ and $\omega_p$ are respectively the archimedean and $p$-adic densities. Loughran \cite[Lemma 2.3]{MR2769338} has shown that we have
\begin{eqnarray*}
\omega_p & = & 1 + \frac{7}{p} + \frac1{p^2} \textrm{.}
\end{eqnarray*}
Let us calculate $\omega_{\infty}$. Let $\mathbf{x} = (x_0,x_1,x_2,x_3)$ and $f(\mathbf{x}) = x_0(x_1+x_2+x_3)^2 - x_1x_2x_3 $. We parametrize the points of $V$ with $x_1$, $x_2$ and $x_3$. We have
\begin{eqnarray*}
\frac{\partial f}{\partial x_0}(\mathbf{x}) & = & (x_1+x_2+x_3)^2 \textrm{,}
\end{eqnarray*}
and since $\mathbf{x} = - \mathbf{x} \in \mathbb{P}^3$, we obtain
\begin{eqnarray*}
\omega_{\infty} & = & \frac1{2} \int \int \int_{\left| x_1 x_2 x_3 \right| / (x_1+x_2+x_3)^2, |x_1|, |x_2|, |x_3| \leq 1} \
\frac{\D x_1 \D x_2 \D x_3}{(x_1+x_2+x_3)^2} \textrm{.}
\end{eqnarray*}
Recall the definition \eqref{equation h} of the function $h$. The change of variables defined by $x_1 = t^2x$, $x_2 = t^2y$ and $x_3 = - t^2(x + y - t)$ yields
\begin{eqnarray}
\notag
\omega_{\infty} & = & \frac{3}{2} \int \int \int_{h(x,y,t) \leq 1} \D x \D y \D t  \\
\label{omega}
& = & 3 \int \int \int_{t>0, h(x,y,t) \leq 1} \D x \D y \D t  \textrm{.}
\end{eqnarray}

\section{Proof of the main theorem}

\label{proof section}

\subsection{Restriction of the domain}

\label{restriction}

Note that, in the torsor equation \eqref{torsor}, the first three terms are less than $B / \eta_1^2 \eta_2 \eta_3 \eta_4$ (by the height conditions \eqref{height2}, \eqref{height3}  and \eqref{height4}), and thus we have
\begin{eqnarray*}
\eta_1^3 \eta_2^2 \eta_3^2 \eta_4^2 \eta_5 \eta_6 \eta_7 & \leq & 3B \textrm{.}
\end{eqnarray*}

From now on, for $n \in \mathbb{Z}_{\geq 1}$, we denote by $\sq(n)$ the unique positive integer such that $\sq(n)^2 | n$ and $n / \sq(n)^2$ is squarefree. Note that for two coprime integers $m,n  \in \mathbb{Z}_{\geq 1}$, we have $\sq(mn) = \sq(m)\sq(n)$.

We now need to show that we can assume along the proof that
\begin{eqnarray}
\label{assumption1}
\eta_1 \sq(\eta_2 \eta_3 \eta_4) & \geq & B^{15 / \log(\log(B))} \textrm{,}
\end{eqnarray}
and, in addition, that
\begin{eqnarray}
\label{assumption2}
\eta_1^3 \eta_2^2 \eta_3^2 \eta_4^2 \eta_5 \eta_6 \eta_7 & \geq & \frac{B}{\log(\log(B))} \textrm{.}
\end{eqnarray}

The proof of lemma \ref{T} shows that we can make use of the estimates of Browning given in \cite[Section~$6$]{MR2250046} to prove that the contributions to $N_{U,H}(B)$ coming from those $(\eta_1, \dots ,\eta_{10}) \in \mathcal{T}(B)$ which do not satisfy one of the two inequalities \eqref{assumption1} and \eqref{assumption2} are actually negligible.

We start by proving the following lemma.

\begin{lemma}
\label{lemmaloglog1}
Let $\mathcal{M}(B)$ be the overall contribution to $N_{U,H}(B)$ coming from those $(\eta_1, \dots ,\eta_{10}) \in \mathcal{T}(B)$ such that
$\eta_1 \sq(\eta_2 \eta_3 \eta_4) \leq B^{15 / \log(\log(B))}$. We have
\begin{eqnarray*}
\mathcal{M}(B) & \ll & \frac{B \log(B)^6}{\log(\log(B))} \textrm{.}
\end{eqnarray*}
\end{lemma}

\begin{proof}
Recall the notation introduced in the proof of lemma \ref{T}. We note that the condition $\eta_1 \sq(\eta_2 \eta_3 \eta_4) \leq B^{15 / \log(\log(B))}$ is equivalent to
$\eta_1'  \leq B^{15 / \log(\log(B))}$.

For $i = 1, \dots, 10$, we let $Y_i$ be variables running over the set $\{ 2^n, n\geq -1\}$. By counting the number of $(\eta_1', \dots, \eta_{10}') \in \mathcal{T}'(B)$ which are subject to $Y_i < |\eta_i'| \leq 2Y_i$, for $i = 1, \dots, 10$, using the work of Browning \cite[Sections $6.1$, $6.2$]{MR2250046}, we obtain that
\begin{eqnarray}
\label{M(B)}
\ \ \ \ \ \ \ \mathcal{M}(B) & \ll & B \log(B)^5 + \sum_{Y_i} X_0^{1/2} X_1^{1/6}  X_2^{1/6}  X_3^{1/6} + \sum_{Y_i} X_0^{1/4} X_1^{1/4}  X_2^{1/4}  X_3^{1/4} \\
\notag
& & + \sum_{Y_i} \max_{\{i,j,k\} = \{2,3,4\}} \left\{
\frac{Y_1 Y_2 Y_3 Y_4 Y_5 Y_6 Y_7 Y_8 Y_9 Y_{10}}{Y_{k+6} \max \left\{ Y_i Y_{i+3}^2 Y_{i+6}, Y_j Y_{j+3}^2 Y_{j+6}, Z_k \right\}} \right\} \textrm{,}
\end{eqnarray}
where the three sums are over the $Y_i$, $i = 1, \dots, 10$, subject to the inequalities
\begin{eqnarray}
\label{heighta}
Y_8 Y_9 Y_{10} & \leq & B \textrm{,} \\
\label{heightb}
Y_1^2 Y_2^2 Y_3 Y_4 Y_5^2 Y_8 & \leq & B \textrm{,} \\
\label{heightc}
Y_1^2 Y_2 Y_3^2 Y_4 Y_6^2 Y_9 & \leq & B \textrm{,} \\
\label{heightd}
Y_1^2 Y_2 Y_3 Y_4^2 Y_7^2 Y_{10} & \leq & B \textrm{,}
\end{eqnarray}
and also
\begin{eqnarray}
\label{extraY_1}
Y_1 & \leq & B^{15 / \log(\log(B))} \textrm{,}
\end{eqnarray}
and where $X_0$, $X_1$, $X_2$, $X_3$ respectively denote the left-hand sides of the inequalities \eqref{heighta}, \eqref{heightb}, \eqref{heightc} and \eqref{heightd} and finally, for $k \in \{2,3,4\}$, $Z_k$ is defined by
\begin{eqnarray*}
Z_k & = &
\begin{cases}
Y_k Y_{k+3}^2 Y_{k+6} & \textrm{ if }  Y_k Y_{k+3}^2 Y_{k+6} \geq Y_1 Y_2 Y_3 Y_4 Y_5 Y_6 Y_7 \textrm{,} \\
1 & \textrm{ otherwise.}
\end{cases}
\end{eqnarray*}

Let us denote by $\mathcal{N}_1(B)$, $\mathcal{N}_2(B)$ and $\mathcal{N}_3(B)$ the respective contributions of the three sums in \eqref{M(B)}. In the following estimations, the notation $\sum_{\widehat{Y_j}}$ indicates that the summation is over all the $Y_i$ with $i \neq j$. We start by investigating the quantity $\mathcal{N}_1(B)$ by summing over $Y_5$, $Y_6$ and $Y_7$ using respectively the conditions \eqref{heightb}, \eqref{heightc} and \eqref{heightd}. We get
\begin{eqnarray*}
\mathcal{N}_1(B) & = & \sum_{Y_i} Y_1 Y_2^{2/3} Y_3^{2/3} Y_4^{2/3} Y_5^{1/3} Y_6^{1/3} Y_7^{1/3} Y_8^{2/3} Y_9^{2/3} Y_{10}^{2/3}  \\
& \ll & B^{1/2} \sum_{\widehat{Y_5}, \widehat{Y_6}, \widehat{Y_7}} Y_8^{1/2} Y_9^{1/2} Y_{10}^{1/2}  \\
& \ll & B \sum_{\widehat{Y_5}, \widehat{Y_6}, \widehat{Y_7}, \widehat{Y_8}} 1 \\
& \ll & \frac{B \log(B)^6}{\log(\log(B))}  \textrm{,}
\end{eqnarray*}
where we have used the condition \eqref{heighta} to sum over $Y_8$ and the condition \eqref{extraY_1} to sum over $Y_1$.  We now deal with the quantity $\mathcal{N}_2(B)$ in a similar fashion. We use exactly the same order of summation and the same inequalities. We obtain
\begin{eqnarray*}
\mathcal{N}_2(B) & = & \sum_{Y_i} Y_1^{3/2} Y_2 Y_3 Y_4 Y_5^{1/2} Y_6^{1/2} Y_7^{1/2} Y_8^{1/2} Y_9^{1/2} Y_{10}^{1/2}  \\
& \ll & B^{3/4} \sum_{\widehat{Y_5}, \widehat{Y_6}, \widehat{Y_7}} Y_8^{1/4} Y_9^{1/4} Y_{10}^{1/4}  \\
& \ll & B \sum_{\widehat{Y_5}, \widehat{Y_6}, \widehat{Y_7}, \widehat{Y_8}} 1 \\
& \ll & \frac{B \log(B)^6}{\log(\log(B))}  \textrm{.}
\end{eqnarray*}
We now deal with the quantity $\mathcal{N}_3(B)$. We only treat the case where $(i,j,k)=(2,3,4)$ since they are all identical. Note that if $Z_4 = Y_4 Y_7^2 Y_{10}$ then $\mathcal{N}_3(B) \leq \mathcal{N}_1(B)$. Thus, we only need to deal with the case where $Z_4 = 1$. In addition, we proceed without loss of generality under the assumption that $Y_2 Y_5^2 Y_8 \leq Y_3 Y_6^2 Y_9$. We first use this condition to sum over $Y_5$ and then we sum over $Y_7$ and $Y_8$ using respectively the conditions \eqref{heightd} and \eqref{heighta}. We get
\begin{eqnarray*}
\mathcal{N}_3(B) & \ll & \sum_{Y_i} Y_1 Y_2 Y_4 Y_5Y_6^{-1} Y_7 Y_8  \\
& \ll & \sum_{\widehat{Y_5}} Y_1 Y_2^{1/2} Y_3^{1/2} Y_4 Y_7 Y_8^{1/2} Y_9^{1/2} \\
& \ll & B^{1/2} \sum_{\widehat{Y_5}, \widehat{Y_7}}  Y_8^{1/2} Y_9^{1/2} Y_{10}^{-1/2} \\
& \ll & B \sum_{\widehat{Y_5}, \widehat{Y_7}, \widehat{Y_8}} Y_{10}^{-1} \\
& \ll & \frac{B \log(B)^6}{\log(\log(B))} \textrm{,}
\end{eqnarray*}
which completes the proof of lemma \ref{lemmaloglog1}.
\end{proof}

The following lemma proves that the contribution to $N_{U,H}(B)$ coming from those $(\eta_1, \dots ,\eta_{10}) \in \mathcal{T}(B)$ which are subject to the stronger condition
\begin{eqnarray*}
\eta_1^3 \eta_2^2 \eta_3^2 \eta_4^2 \eta_5 \eta_6 \eta_7 & \leq & \frac{B}{\log(\log(B))} \textrm{,}
\end{eqnarray*}
is negligible.

\begin{lemma}
\label{lemmaloglog2}
Let $\mathcal{M}'(B)$ be the overall contribution to $N_{U,H}(B)$ coming from those $(\eta_1, \dots ,\eta_{10}) \in \mathcal{T}(B)$ such that
\begin{eqnarray*}
\eta_1^3 \eta_2^2 \eta_3^2 \eta_4^2 \eta_5 \eta_6 \eta_7 & \leq & \frac{B}{\log(\log(B))} \textrm{.}
\end{eqnarray*}
We have
\begin{eqnarray*}
\mathcal{M}'(B) & \ll & \frac{B \log(B)^6}{\log(\log(B))^{1/6}} \textrm{.}
\end{eqnarray*}
\end{lemma}

\begin{proof}
We proceed as in the proof of lemma \ref{lemmaloglog1} and we use the same notations. We have
\begin{eqnarray}
\label{M'(B)}
\ \ \ \ \ \ \ \mathcal{M}'(B) & \ll & B \log(B)^5 + \sum_{Y_i} X_0^{1/2} X_1^{1/6}  X_2^{1/6}  X_3^{1/6} + \sum_{Y_i} X_0^{1/4} X_1^{1/4}  X_2^{1/4}  X_3^{1/4} \\
\notag
& & + \sum_{Y_i} \max_{\{i,j,k\} = \{2,3,4\}} \left\{
\frac{Y_1 Y_2 Y_3 Y_4 Y_5 Y_6 Y_7 Y_8 Y_9 Y_{10}}{Y_{k+6} \max \left\{ Y_i Y_{i+3}^2 Y_{i+6}, Y_j Y_{j+3}^2 Y_{j+6}, Z_k \right\} } \right\} \textrm{,}
\end{eqnarray}
where the three sums are over the dyadic variables $Y_i$,  $i = 1, \dots, 10 $, subject to the inequalities \eqref{heighta}, \eqref{heightb}, \eqref{heightc}, \eqref{heightd} and
\begin{eqnarray}
\label{extra}
Y_1^3 Y_2^2 Y_3^2 Y_4^2 Y_5 Y_6 Y_7 & \leq & \frac{B}{\log(\log(B))} \textrm{.}
\end{eqnarray}
Let us denote by $\mathcal{N}'_1(B)$, $\mathcal{N}'_2(B)$ and $\mathcal{N}'_3(B)$ the respective contributions of the three sums in \eqref{M'(B)}. Combining the conditions \eqref{heighta} and \eqref{heightb}, we get
\begin{eqnarray}
\label{combined}
Y_1^{1/4} Y_2^{1/4} Y_3^{1/8} Y_4^{1/8} Y_5^{1/4} Y_8 Y_9^{7/8} Y_{10}^{7/8} & \leq & B \textrm{.}
\end{eqnarray}
We start by bounding the contribution of the quantity $\mathcal{N}'_1(B)$ by summing successively over $Y_8$, $Y_9$ and $Y_{10}$ using respectively the conditions \eqref{combined}, \eqref{heightc} and \eqref{heightd}. We deduce
\begin{eqnarray*}
\mathcal{N}'_1(B) & = & \sum_{Y_i} Y_1 Y_2^{2/3} Y_3^{2/3} Y_4^{2/3} Y_5^{1/3} Y_6^{1/3} Y_7^{1/3} Y_8^{2/3} Y_9^{2/3} Y_{10}^{2/3}  \\
& \ll & B^{2/3} \sum_{\widehat{Y_8}} Y_1^{5/6} Y_2^{1/2} Y_3^{7/12} Y_4^{7/12} Y_5^{1/6} Y_6^{1/3} Y_7^{1/3} Y_9^{1/12} Y_{10}^{1/12} \\
& \ll & B^{5/6} \sum_{\widehat{Y_8}, \widehat{Y_9}, \widehat{Y_{10}}} Y_1^{1/2} Y_2^{1/3} Y_3^{1/3} Y_4^{1/3} Y_5^{1/6} Y_6^{1/6} Y_7^{1/6} \\
& \ll & \frac{B}{\log(\log(B))^{1/6}} \sum_{\widehat{Y_7}, \widehat{Y_8}, \widehat{Y_9}, \widehat{Y_{10}}} 1 \\
& \ll & \frac{B \log(B)^6}{\log(\log(B))^{1/6}} \textrm{,}
\end{eqnarray*}
where we have summed over $Y_7$ using the condition \eqref{extra}. Let us now turn to the estimation of $\mathcal{N}'_2(B)$. We successively sum over $Y_8$, $Y_9$ and $Y_{10}$ using respectively the conditions \eqref{combined}, \eqref{heightc} and \eqref{heightd}. We obtain
\begin{eqnarray*}
\mathcal{N}'_2(B) & = & \sum_{Y_i} Y_1^{3/2} Y_2 Y_3 Y_4 Y_5^{1/2} Y_6^{1/2} Y_7^{1/2} Y_8^{1/2} Y_9^{1/2} Y_{10}^{1/2}  \\
& \ll & B^{1/2} \sum_{\widehat{Y_8}} Y_1^{11/8} Y_2^{7/8} Y_3^{15/16} Y_4^{15/16} Y_5^{3/8} Y_6^{1/2} Y_7^{1/2} Y_9^{1/16} Y_{10}^{1/16} \\
& \ll & B^{5/8} \sum_{\widehat{Y_8}, \widehat{Y_9}, \widehat{Y_{10}}} Y_1^{9/8} Y_2^{3/4} Y_3^{3/4} Y_4^{3/4} Y_5^{3/8} Y_6^{3/8} Y_7^{3/8}\\
& \ll & \frac{B}{\log(\log(B))^{3/8}} \sum_{\widehat{Y_7}, \widehat{Y_8}, \widehat{Y_9}, \widehat{Y_{10}}} 1 \\
& \ll & \frac{B \log(B)^6}{\log(\log(B))^{3/8}} \textrm{,}
\end{eqnarray*}
where we have summed over $Y_7$ using the condition \eqref{extra}. We now turn to the case of the quantity $\mathcal{N}'_3(B)$ and, as in the proof of lemma \ref{lemmaloglog1}, we only treat the case where $(i,j,k)=(2,3,4)$ and we work under the assumptions that $Z_4=1$ and thus
\begin{eqnarray}
\label{Z}
Y_4 Y_7^2 Y_{10} & \leq & Y_1 Y_2 Y_3 Y_4 Y_5 Y_6 Y_7 \textrm{,}
\end{eqnarray}
and $Y_2 Y_5^2 Y_8 \leq Y_3 Y_6^2 Y_9$. Combining the conditions \eqref{extra} and \eqref{Z}, we get
\begin{eqnarray}
\label{combinedZ}
Y_1^2 Y_2 Y_3 Y_4^2 Y_7^2 Y_{10} & \leq & \frac{B}{\log(\log(B))}  \textrm{.}
\end{eqnarray}
We first use the condition $Y_2 Y_5^2 Y_8 \leq Y_3 Y_6^2 Y_9$ to sum over $Y_5$ and then we sum over $Y_8$ and $Y_7$ using respectively the conditions \eqref{heighta} and \eqref{combinedZ}. We deduce
\begin{eqnarray*}
\mathcal{N}_3'(B) & \ll & \sum_{Y_i} Y_1 Y_2 Y_4 Y_5 Y_6^{-1} Y_7 Y_8 \\
& \ll & \sum_{\widehat{Y_5}} Y_1 Y_2^{1/2} Y_3^{1/2} Y_4 Y_7 Y_8^{1/2} Y_9^{1/2} \\
& \ll & B^{1/2} \sum_{\widehat{Y_5}, \widehat{Y_8}} Y_1 Y_2^{1/2} Y_3^{1/2} Y_4 Y_7 Y_{10}^{-1/2} \\
& \ll & \frac{B}{\log(\log(B))^{1/2}} \sum_{\widehat{Y_5}, \widehat{Y_7}, \widehat{Y_8}} Y_{10}^{-1} \\
& \ll & \frac{B \log(B)^6}{\log(\log(B))^{1/2}} \textrm{,}
\end{eqnarray*}
which completes the proof of lemma \ref{lemmaloglog2}.
\end{proof}

\subsection{Setting up}

First, we recall that we have the following condition (given at the beginning of section \ref{restriction}),
\begin{eqnarray}
\label{height new}
\eta_1^3 \eta_2^2 \eta_3^2 \eta_4^2 \eta_5 \eta_6 \eta_7 & \leq & 3 B \textrm{.}
\end{eqnarray}

It is easy to check that the symmetry between the three quantities $\eta_2\eta_5^2$, $\eta_3\eta_6^2$ and $\eta_4\eta_7^2$ is demonstrated by the action of $\mathfrak{S}_3$ on $\left\{ (\eta_2,\eta_5,\eta_8), (\eta_3,\eta_6,\eta_9), (\eta_4,\eta_7,\eta_{10}) \right\}$. Along the proof, we will assume that
\begin{eqnarray*}
\eta_4\eta_7^2 & \leq & \eta_2\eta_5^2, \eta_3\eta_6^2 \textrm{.}
\end{eqnarray*}
The following lemma proves that we just need to multiply our future main term by a factor $3$ to take this new assumption into account.

\begin{lemma}
\label{equal}
Let $N_0(B)$ be the total number of $(\eta_1, \dots, \eta_{10}) \in \mathcal{T}(B)$ such that $\eta_2 \eta_5^2 = \eta_4 \eta_7^2$ or $\eta_3 \eta_6^2 = \eta_4 \eta_7^2$. We have the upper bound
\begin{eqnarray*}
N_0(B) & \ll & B \log(B)^3 \textrm{.}
\end{eqnarray*}
\end{lemma}

\begin{proof}
By symmetry, we only need to treat the case of the condition $\eta_3 \eta_6^2 = \eta_4 \eta_7^2$. This equality and the condition $\gcd(\eta_3 \eta_6, \eta_4 \eta_7) = 1$ imply that
$\eta_3 = \eta_4 = \eta_6 = \eta_7 = 1$. In this situation, the torsor equation is simply
\begin{eqnarray*}
\eta_2 \eta_5^2 \eta_8 + \eta_9 + \eta_{10} - \eta_1 \eta_2 \eta_5 & = & 0 \textrm{.}
\end{eqnarray*}
Thus, $N_0(B)$ is bounded by the number of $(\eta_1, \eta_2, \eta_5, \eta_8, \eta_9) \in \mathbb{Z}_{>0}^3 \times \mathbb{Z}_{\neq 0}^2$ satisfying
\begin{eqnarray*}
|\eta_8 \eta_9| \left| \eta_2 \eta_5^2 \eta_8 + \eta_9 - \eta_1 \eta_2 \eta_5 \right| & \leq & B \textrm{,} \\
\eta_1^2 \eta_2^2\eta_5^2 |\eta_8| & \leq & B \textrm{.}
\end{eqnarray*}
Using \cite[Lemma $1$]{2A2+A1} to count the number of $\eta_9$ satisfying the first of these two inequalities, we obtain
\begin{eqnarray*}
N_0(B) & \ll & \sum_{\substack{\eta_1, \eta_2, \eta_5, \eta_8 \\ \eta_1^2 \eta_2^2\eta_5^2 |\eta_8| \leq B}} \left( \frac{B^{1/2}}{|\eta_8|^{1/2}} + 1 \right) \\
& \ll & B \log(B)^3 \textrm{,}
\end{eqnarray*}
as wished.
\end{proof}

Let $N(B)$ be the overall contribution of those $(\eta_1, \dots, \eta_{10}) \in \mathcal{T}(B)$ subject to the conditions
\begin{eqnarray}
\label{symmetry}
\eta_4\eta_7^2 & \leq & \eta_2\eta_5^2, \eta_3\eta_6^2 \textrm{,} \\
\label{log}
B^{15 / \log(\log(B))} & \leq & \eta_1  \sq(\eta_2 \eta_3 \eta_4) \textrm{,} \\
\label{loglog}
\frac{B}{\log(\log(B))} & \leq & \eta_1^3 \eta_2^2 \eta_3^2 \eta_4^2 \eta_5 \eta_6 \eta_7 \textrm{.}
\end{eqnarray}
Lemmas \ref{T}, \ref{lemmaloglog1}, \ref{lemmaloglog2} and \ref{equal} give us the following result.

\begin{lemma}
\label{N(B)}
We have the estimate
\begin{eqnarray*}
N_{U,H}(B) & = & 3 N(B) + O \left( \frac{B \log(B)^6}{\log(\log(B))^{1/6}} \right) \textrm{.}
\end{eqnarray*}
\end{lemma}

The end of the proof is devoted to the estimation of $N(B)$.

\subsection{Application of lemma \ref{lemma affine final}}

The idea of the proof is to view the equation \eqref{torsor} as a congruence modulo $\eta_4 \eta_7^2$.  For this, we replace $\eta_{10}$ by its value given by the equation \eqref{torsor} in the height conditions \eqref{height1} and \eqref{height4}. These conditions become
\begin{eqnarray*}
|\eta_8 \eta_9| \left| \eta_2 \eta_5^2 \eta_8 + \eta_3 \eta_6^2 \eta_9 - \eta_1 \eta_2 \eta_3 \eta_4 \eta_5 \eta_6 \eta_7 \right| & \leq & B \eta_4 \eta_7^2 \textrm{,} \\
\eta_1^2 \eta_2 \eta_3 \eta_4 \left| \eta_2 \eta_5^2 \eta_8 + \eta_3 \eta_6^2 \eta_9 - \eta_1 \eta_2 \eta_3 \eta_4 \eta_5 \eta_6 \eta_7 \right| & \leq & B \textrm{,}
\end{eqnarray*}
and we keep denoting them respectively by \eqref{height1} and \eqref{height4}. From now on, we use the notation
$\boldsymbol{\eta} = (\eta_2, \eta_3, \eta_4, \eta_5, \eta_6, \eta_7)$ and we set
\begin{eqnarray*}
\boldsymbol{\eta}^{(r_2,r_3,r_4,r_5,r_6,r_7)} & = &  \eta_2^{r_2} \eta_3^{r_3} \eta_4^{r_4} \eta_5^{r_5} \eta_6^{r_6} \eta_7^{r_7} \textrm{,}
\end{eqnarray*}
for $(r_2,r_3,r_4,r_5,r_6,r_7) \in \mathbb{Q}^6$. We set
\begin{eqnarray}
\notag
Y & = & \frac{B}{\eta_2 \eta_3 \eta_4} \textrm{,} \\
\label{Z_1}
Z_1 & = & \frac{B^{1/3}}{\boldsymbol{\eta}^{(2/3,2/3,2/3,1/3,1/3,1/3)}} \textrm{,}
\end{eqnarray}
and, for brevity, $q_8 = \eta_2 \eta_5^2$, $q_9 = \eta_3 \eta_6^2$, $q_{10} = \eta_4 \eta_7^2$. It is immediate to check that $\boldsymbol{\eta}$ is restricted to lie in the region $\mathcal{V}$ defined by
\begin{eqnarray}
\label{V}
\ \ \ \ \ \ \ \ \mathcal{V} & = & \left\{ \boldsymbol{\eta} \in \mathbb{Z}_{> 0}^6,
\begin{array}{l}
Y \log(\log(B))^{2/3} \geq q_8 Z_1^2, Y \log(\log(B))^{2/3} \geq q_9 Z_1^2 \\
Z_1 \geq 3^{-1/3}, q_8 \geq q_{10}, q_9 \geq q_{10}
\end{array}
\right\} \textrm{.}
\end{eqnarray}
We consider that $\eta_1 \in  \mathbb{Z}_{> 0}$ and $\boldsymbol{\eta} \in \mathcal{V}$ are fixed and are subject to the conditions \eqref{height new}, \eqref{log} and \eqref{loglog} and to the coprimality conditions \eqref{gcd4}, \eqref{gcd5} and \eqref{gcd6}. Let $N(\eta_1, \boldsymbol{\eta},B)$ be the number of $(\eta_8, \eta_9, \eta_{10}) \in \mathbb{Z}_{\neq 0}^3$ satisfying the equation \eqref{torsor}, the height conditions \eqref{height1}, \eqref{height2}, \eqref{height3}, \eqref{height4} and finally the coprimality conditions \eqref{gcd1}, \eqref{gcd2} and \eqref{gcd3}. The goal of this section is to prove the following lemma.

\begin{lemma}
\label{lemma inter}
We have the estimate
\begin{eqnarray*}
N(\eta_1, \boldsymbol{\eta},B) & = & \frac{B^{2/3}}{\boldsymbol{\eta}^{(1/3,1/3,1/3,2/3,2/3,2/3)}}
g_2 \left( \frac{\eta_1}{Z_1} \right) \theta_1(\eta_1, \boldsymbol{\eta}) \theta_2(\boldsymbol{\eta}) + R(\eta_1, \boldsymbol{\eta},B) \textrm{,}
\end{eqnarray*}
where $\theta_1(\eta_1, \boldsymbol{\eta})$ and $\theta_2(\boldsymbol{\eta})$ are arithmetic functions respectively defined in \eqref{theta1} and \eqref{theta2} and
\begin{eqnarray*}
\sum_{\eta_1, \boldsymbol{\eta}} R(\eta_1, \boldsymbol{\eta},B) & \ll & B \log(B)^5\log(\log(B))^{7/3} \textrm{.}
\end{eqnarray*}
\end{lemma}

First, we see that since $\gcd(\eta_2\eta_5, \eta_3\eta_6\eta_9) = 1$ and $\gcd(\eta_3\eta_6, \eta_2\eta_5\eta_8) = 1$, the equation \eqref{torsor} proves that the coprimality condition \eqref{gcd1} can be replaced by $\gcd(\eta_{10},\eta_1\eta_4) = 1$. Let us remove the coprimality conditions $\gcd(\eta_8,\eta_6)=1$ and $\gcd(\eta_9,\eta_5)=1$ using Möbius inversions, we obtain
\begin{eqnarray*}
N(\eta_1, \boldsymbol{\eta},B) & = & \sum_{\substack{k_8 | \eta_6 \\ \gcd(k_8, \eta_1\eta_2\eta_3\eta_4\eta_7) = 1}} \mu(k_8) \sum_{\substack{k_9 | \eta_5 \\ \gcd(k_9, \eta_1\eta_2\eta_3\eta_4\eta_7) = 1}} \mu(k_9) S_{k_8,k_9}(\eta_1, \boldsymbol{\eta},B) \textrm{,}
\end{eqnarray*}
where
\begin{eqnarray*}
S_{k_8,k_9}(\eta_1, \boldsymbol{\eta},B) & = & \# \left\{ \left( \eta_8', \eta_9', \eta_{10} \right) \in \mathbb{Z}_{\neq 0}^3,
\begin{array}{l}
\eta_2 \eta_5^2 k_8 \eta_8' + \eta_3 \eta_6^2 k_9 \eta_9' + \eta_4 \eta_7^2 \eta_{10} = b \\
\eqref{height1}, \eqref{height2}, \eqref{height3}, \eqref{height4} \\
\gcd(\eta_{10},\eta_1\eta_4)=1 \\
\gcd(\eta_8'\eta_9',\eta_1\eta_2\eta_3\eta_4\eta_7) = 1 \\
\end{array}
\right\} \textrm{,}
\end{eqnarray*}
where we use the notations $\eta_8=k_8\eta_8'$ and $\eta_9=k_9\eta_9'$ and $b = \eta_1 \eta_2 \eta_3 \eta_4 \eta_5 \eta_6 \eta_7$. From now on, we set
\begin{eqnarray*}
\mathcal{Z} & = & B^{1/\log(\log(B))} \textrm{.}
\end{eqnarray*}
To take care of the error terms showing up in the application of lemma \ref{lemma affine final}, we need to show that the summations over $k_8$ and $k_9$ can be restricted to $k_8, k_9 \leq \mathcal{Z}^3$. To do so, let $N'(\eta_1, \boldsymbol{\eta},B) $ be the contribution of $N(\eta_1, \boldsymbol{\eta},B)$ under the assumption $k_8 > \mathcal{Z}^3$, that is to say
\begin{eqnarray*}
N'(\eta_1, \boldsymbol{\eta},B) & = & \sum_{\substack{k_8 | \eta_6, k_8 > \mathcal{Z}^3  \\ \gcd(k_8, \eta_1\eta_2\eta_3\eta_4\eta_7) = 1}}  \ \ \ \sum_{\substack{k_9 | \eta_5  \\ \gcd(k_9, \eta_1\eta_2\eta_3\eta_4\eta_7) = 1}} S_{k_8,k_9}(\eta_1, \boldsymbol{\eta},B) \textrm{.}
\end{eqnarray*}
Let us write $\eta_6 = k_8 \eta_6'$ and $\eta_5 = k_9 \eta_5'$. We notice that the equation implies that $k_8 k_9 | \eta_{10}$ and thus we also write $\eta_{10} = k_8 k_9 \xi_{10}$. With these notations, we get
\begin{eqnarray*}
N'(\eta_1, \boldsymbol{\eta},B) & = & \sum_{\substack{\mathcal{Z}^3 < k_8 \leq B^{1/2} \\ \gcd(k_8, \eta_1\eta_2\eta_3\eta_4\eta_7) = 1}}  \ \ \ \sum_{\substack{k_9 \leq B^{1/2} \\ \gcd(k_9, \eta_1\eta_2\eta_3\eta_4\eta_7) = 1}} S_{k_8,k_9}'(\eta_1, \boldsymbol{\eta},B) \textrm{,}
\end{eqnarray*}
where
\begin{eqnarray*}
 S_{k_8,k_9}'(\eta_1, \boldsymbol{\eta},B) & \! \! \!  = \! \! \! & \# \left\{ \left( \eta_8', \eta_9', \xi_{10} \right) \in \mathbb{Z}_{\neq 0}^3,
\begin{array}{l}
\eta_2 \eta_5'^2 k_9 \eta_8' + \eta_3 \eta_6'^2 k_8 \eta_9' + \eta_4 \eta_7^2 \xi_{10} = b' \\
\eqref{height1}, \eqref{height2}, \eqref{height3}, \eqref{height4} \\
\gcd(\xi_{10},\eta_1\eta_4)=1 \\
\gcd(\eta_8'\eta_9',\eta_1\eta_2\eta_3\eta_4\eta_7) = 1 \\
\end{array}
\right\} \textrm{,}
\end{eqnarray*}
where we have set $b' = \eta_1 \eta_2 \eta_3 \eta_4 \eta_5' \eta_6' \eta_7$. Let us split the summations over $k_8$ and $k_9$ into dyadic ranges. Let us assume that $K_8, K_9 \geq 1/2$ and that $K_8 < k_8 \leq 2K_8$ and $K_9 < k_9 \leq 2K_9$. Let us set $\xi_8 = k_9 \eta_8'$ and $\xi_9 = k_8 \eta_9'$. The height conditions \eqref{height1}, \eqref{height2}, \eqref{height3} and \eqref{height4} imply respectively
\begin{eqnarray}
\label{height1K}
|\xi_8 \xi_9 \xi_{10}| & \leq & \frac{B}{K_8K_9} \textrm{,} \\
\label{height2K}
\eta_1^2 \eta_2^2 \eta_3 \eta_4 \eta_5'^2 |\xi_8| & \leq & \frac{B}{K_8K_9} \textrm{,} \\
\label{height3K}
\eta_1^2 \eta_2 \eta_3^2 \eta_4 \eta_6'^2 |\xi_9| & \leq & \frac{B}{K_8K_9} \textrm{,} \\
\label{height4K}
\eta_1^2 \eta_2 \eta_3 \eta_4^2 \eta_7^2 |\xi_{10}| & \leq & \frac{B}{K_8K_9} \textrm{.}
\end{eqnarray}
We thus have, for $K_8 < k_8 \leq 2K_8$ and $K_9 < k_9 \leq 2K_9$,
\begin{eqnarray*}
S_{k_8,k_9}'(\eta_1, \boldsymbol{\eta},B) & \ll &
\# \left\{ \left( \xi_8, \xi_9, \xi_{10} \right) \in \mathbb{Z}_{\neq 0}^3,
\begin{array}{l}
k_8 | \xi_9, k_9 | \xi_8 \\
\eta_2 \eta_5'^2 \xi_8 + \eta_3 \eta_6'^2 \xi_9 + \eta_4 \eta_7^2 \xi_{10} = b' \\
\eqref{height1K}, \eqref{height2K}, \eqref{height3K}, \eqref{height4K} \\
\gcd(\xi_{10},\eta_1\eta_4)=1 \\
\gcd(\xi_8\xi_9,\eta_1\eta_2\eta_3\eta_4\eta_7) = 1 \\
\end{array}
\right\}
\textrm{.}
\end{eqnarray*}
Therefore, using the standard bound for the divisor function,
\begin{eqnarray*}
\tau(n) & \ll & n^{1/\log(\log(3 n))} \textrm{,}
\end{eqnarray*}
for $n \geq 1$, we get
\begin{eqnarray*}
\sum_{\substack{K_8 < k_8 \leq 2K_8 \\ K_9 < k_9 \leq 2K_9}} S_{k_8,k_9}'(\eta_1, \boldsymbol{\eta},B) & \ll &
\mathcal{Z}^2 S_{K_8,K_9} \textrm{,}
\end{eqnarray*}
where $S_{K_8,K_9} = S_{K_8,K_9}(\eta_1,\eta_2,\eta_3,\eta_4,\eta_5',\eta_6',\eta_7,B)$ is defined by
\begin{eqnarray*}
S_{K_8,K_9} & = &
\# \left\{ \left( \xi_8, \xi_9, \xi_{10} \right) \in \mathbb{Z}_{\neq 0}^3,
\begin{array}{l}
\eta_2 \eta_5'^2 \xi_8 + \eta_3 \eta_6'^2 \xi_9+ \eta_4 \eta_7^2 \xi_{10} = b' \\
\eqref{height1K}, \eqref{height2K}, \eqref{height3K}, \eqref{height4K} \\
\gcd(\xi_{10},\eta_1\eta_4)=1 \\
\gcd(\xi_8\xi_9,\eta_1\eta_2\eta_3\eta_4\eta_7) = 1 \\
\end{array}
\right\} \textrm{.}
\end{eqnarray*}
Setting $\xi_{6,8} = \gcd(\eta_6',\xi_8)$ and $\xi_{5,9} = \gcd(\eta_5',\xi_9)$, we see that $\xi_{6,8}\xi_{5,9} |\xi_{10}$, and we thus obtain
\begin{eqnarray*}
\sum_{\eta_1,\eta_2,\eta_3,\eta_4,\eta_5',\eta_6',\eta_7} S_{K_8,K_9} & \ll & \sum_{\xi_{6,8},\xi_{5,9} \leq B} N_{U,H} \left( \frac{B}{K_8 K_9 \xi_{6,8} \xi_{5,9}} \right) \textrm{.}
\end{eqnarray*}
Therefore, we can apply the work of Browning \cite{MR2250046}. We get
\begin{eqnarray*}
\sum_{\eta_1, \boldsymbol{\eta}} N'(\eta_1, \boldsymbol{\eta},B) & \ll & \mathcal{Z}^2 \sum_{\substack{\mathcal{Z}^3 < K_8 < B^{1/2} \\ K_9 < B^{1/2}}}
\sum_{\xi_{6,8},\xi_{5,9} \leq B} \frac{B \log(B)^6}{K_8 K_9 \xi_{6,8} \xi_{5,9}} \\
& \ll & B \mathcal{Z}^{-1/2} \textrm{,}
\end{eqnarray*}
which is satisfactory. Therefore, we can restrict from now on the summations over $k_8$ and $k_9$ as we wished. If we allow $\eta_{10} = 0$ in the definition of the set $S_{k_8,k_9}(\eta_1, \boldsymbol{\eta},B)$ then the coprimality condition $\gcd(\eta_{10},\eta_1\eta_4)=1$ implies $\eta_1 = \eta_4 = 1$. Moreover, the equation
$\eta_2 \eta_5^2 k_8 \eta_8' + \eta_3 \eta_6^2 k_9 \eta_9' = \eta_2 \eta_3 \eta_5 \eta_6 \eta_7$ also implies $\eta_2 = \eta_3 = 1$. These restrictions are in contradiction with the condition \eqref{log}, so from now on, we allow $\eta_{10}$ to vanish in the definition of $S_{k_8,k_9}(\eta_1, \boldsymbol{\eta},B)$. Let us now remove the coprimality condition $\gcd(\eta_{10},\eta_1\eta_4) = 1$ using a Möbius inversion. We get that the main term of $N(\eta_1, \boldsymbol{\eta},B)$ is equal to
\begin{eqnarray*}
\ \sum_{\substack{k_8 | \eta_6, k_8 \leq \mathcal{Z}^3 \\ \gcd(k_8, \eta_1\eta_2\eta_3\eta_4\eta_7) = 1}} \! \! \mu(k_8) \! \! \sum_{\substack{k_9 | \eta_5, k_9 \leq \mathcal{Z}^3 \\ \gcd(k_9, \eta_1\eta_2\eta_3\eta_4\eta_7) = 1}} \! \! \mu(k_9) \sum_{k_{10} | \eta_1\eta_4} \mu(k_{10}) S_{k_8,k_9,k_{10}}(\eta_1, \boldsymbol{\eta},B) \textrm{,}
\end{eqnarray*}
where $S_{k_8,k_9,k_{10}}(\eta_1, \boldsymbol{\eta},B)$ denotes
\begin{eqnarray*}
& & \# \left\{ \left( \eta_8', \eta_9', \eta_{10}' \right) \in \mathbb{Z}_{\neq 0}^2 \times  \mathbb{Z},
\begin{array}{l}
\eta_2 \eta_5^2 k_8 \eta_8' + \eta_3 \eta_6^2 k_9 \eta_9' + \eta_4 \eta_7^2 k_{10} \eta_{10}' = b \\
\eqref{height1}, \eqref{height2}, \eqref{height3}, \eqref{height4} \\
\gcd(\eta_8'\eta_9',\eta_1\eta_2\eta_3\eta_4\eta_7) = 1 \\
\end{array}
\right\} \textrm{.}
\end{eqnarray*}
Since $\gcd(\eta_1 \eta_4, k_8 k_9 \eta_5 \eta_6 \eta_8' \eta_9' )=1$, we have the condition $\gcd(k_{10}, k_8 k_9 \eta_5 \eta_6 \eta_8' \eta_9' )=1$. Moreover, the two conditions $\gcd(\eta_2 \eta_5 k_8 \eta_8', \eta_3) = 1$ and $\gcd(\eta_3 \eta_6 k_9 \eta_9', \eta_2) = 1$ imply that we also have $\gcd(k_{10},\eta_2 \eta_3)=1$. We now remove the coprimality conditions $\gcd(\eta_8' \eta_9',\eta_1 \eta_2 \eta_3) = 1$ using Möbius inversions. Setting $\eta_8' = \ell_8 \eta_8''$ and $\eta_9' = \ell_9 \eta_9''$, we obtain that the main term of $N(\eta_1, \boldsymbol{\eta},B)$ is equal to
\begin{eqnarray*}
& & \sum_{\substack{k_8 | \eta_6, k_8 \leq \mathcal{Z}^3 \\ \gcd(k_8, \eta_1\eta_2\eta_3\eta_4\eta_7) = 1}} \mu(k_8) \sum_{\substack{k_9 | \eta_5, k_9 \leq \mathcal{Z}^3 \\ \gcd(k_9, \eta_1\eta_2\eta_3\eta_4\eta_7) = 1}} \mu(k_9) \\
& & \sum_{\substack{k_{10} | \eta_1\eta_4 \\ \gcd(k_{10}, k_8 k_9 \eta_2 \eta_3 \eta_5 \eta_6) = 1}} \mu(k_{10}) \sum_{\substack{\ell_8, \ell_9 | \eta_1\eta_2\eta_3 \\ \gcd(\ell_8\ell_9,k_{10}\eta_4\eta_7) = 1}} \mu(\ell_8) \mu(\ell_9) S(\eta_1, \boldsymbol{\eta},B) \textrm{,}
\end{eqnarray*}
where $S(\eta_1, \boldsymbol{\eta},B)$ denotes
\begin{eqnarray*}
& & \# \left\{ \left( \eta_8'', \eta_9'' \right) \in \mathbb{Z}_{\neq 0}^2,
\begin{array}{l}
\eta_2 \eta_5^2  k_8 \ell_8 \eta_8'' + \eta_3 \eta_6^2 k_9 \ell_9 \eta_9'' \equiv b \imod{k_{10} \eta_4 \eta_7^2} \\
\eqref{height1}, \eqref{height2}, \eqref{height3}, \eqref{height4} \\
\gcd(\eta_8''\eta_9'',k_{10}\eta_4\eta_7) = 1 \\
\end{array}
\right\} \textrm{.}
\end{eqnarray*}
Note that we have replaced the equation $\eta_2 \eta_5^2 k_8 \ell_8 \eta_8'' + \eta_3 \eta_6^2 k_9 \ell_9 \eta_9'' + \eta_4 \eta_7^2 k_{10} \eta_{10}' = b$ by a congruence. Setting
\begin{eqnarray*}
X & = & \frac{B}{\eta_1^2 \boldsymbol{\eta}^{(1,1,1,0,0,0)}} \textrm{,} \\
T & = & \eta_1 \boldsymbol{\eta}^{(1,1,1,1,1,1)} \textrm{,}
\end{eqnarray*}
and $A_1 = k_8 \ell_8 \eta_2 \eta_5^2$, $A_2 = k_9 \ell_9 \eta_3 \eta_6^2$ and recalling the equality \eqref{equality S}, it is immediate to check that $(\eta_8'', \eta_9'') \in \mathbb{Z}_{\neq 0}^2$ are subject to the height conditions \eqref{height1}, \eqref{height2}, \eqref{height3} and \eqref{height4} if and only if $(\eta_8'', \eta_9'') \in \mathcal{S} \cap \mathbb{Z}_{\neq 0}^2$. Setting $\mathcal{L} = \log(\log(B))$, we see that the condition \eqref{loglog} can be rewritten $X/\mathcal{L} \leq T$. We can therefore apply lemma \ref{lemma affine final} with $L = \log(B)$, $q = k_{10} \eta_4 \eta_7^2$ and $\mathbf{a} = (k_8 \ell_8 \eta_2 \eta_5^2, k_9 \ell_9 \eta_3 \eta_6^2)$. Recall the definitions \eqref{varphi} of $\varphi^{\ast}$ and \eqref{Z_1} of $Z_1$ and also the definitions of $E(q,\mathbf{a})$ and $E_2(q)$ respectively given in lemmas \ref{Ramanujan lemma} and \ref{volume D}. We obtain
\begin{eqnarray*}
S(\eta_1, \boldsymbol{\eta},B) - \frac{\varphi^{\ast}( k_{10} \eta_4 \eta_7)}{k_8 \ell_8 k_9 \ell_9 k_{10}}
\frac{B^{2/3}}{\boldsymbol{\eta}^{(1/3,1/3,1/3,2/3,2/3,2/3)}} g_2 \left( \frac{\eta_1}{Z_1} \right) & \ll &
\mathcal{E} + \mathcal{E}'  \textrm{,}
\end{eqnarray*}
where
\begin{eqnarray*}
\mathcal{E} & = & \log(B)^6 E(q, \mathbf{a})  \textrm{,}
\end{eqnarray*}
and
\begin{eqnarray*}
\mathcal{E}' & = & \frac{B^{2/3}}{k_8 \ell_8 k_9 \ell_9 k_{10} \boldsymbol{\eta}^{(1/3,1/3,1/3,2/3,2/3,2/3)}} \mathcal{L}^{4/3} \\
& & \left( \frac{\mathcal{L}}{\log(B)} + \frac{k_8^{1/2} \ell_8^{1/2} \eta_1 \eta_2 \eta_3^{1/2} \eta_4^{1/2} \eta_5}{B^{1/2}}
+ \frac{k_9^{1/2} \ell_9^{1/2} \eta_1 \eta_2^{1/2} \eta_3 \eta_4^{1/2} \eta_6}{B^{1/2}} \right) E_2(q) \textrm{.}
\end{eqnarray*}
Let us estimate the contribution of these error terms. Let us start by bounding the overall contribution of $\mathcal{E}$. For this, we write $\eta_5 = k_9 \eta_5'$ and $\eta_6 = k_8 \eta_6'$, and we let $Y_5$, $Y_6$ and $Y_7$ be variables running over the set $\{ 2^n, n\geq -1\}$. We define $\mathcal{N} = \mathcal{N}(Y_5, Y_6, Y_7)$ as the sum over $\eta_5', \eta_6', \eta_7 \in \mathbb{Z}_{\geq 1}$ satisfying $Y_5 < k_9 \eta_5' \leq 2 Y_5$, $Y_6 < k_8 \eta_6' \leq 2Y_6$ and $Y_7 < \eta_7 \leq 2Y_7$ and the coprimality conditions $\gcd(\eta_5'\eta_6', \eta_4 \eta_7) = 1$ and $\gcd(\eta_5',\eta_6') = 1$, of the quantity
\begin{eqnarray*}
\sum_{\substack{k_8,k_9 \leq \mathcal{Z}^3 \\ \gcd(k_8k_9, \eta_1\eta_2\eta_3\eta_4\eta_7) = 1}} \ \  \sum_{\substack{k_{10} | \eta_1\eta_4 \\ \gcd(k_{10}, k_8 k_9 \eta_2 \eta_3 \eta_5' \eta_6') = 1}} \ \
\sum_{\substack{\ell_8, \ell_9 | \eta_1\eta_2\eta_3 \\ \gcd(\ell_8\ell_9,k_{10}\eta_4\eta_7) = 1}} \log(B)^6 E(q, \mathbf{a}') \textrm{,}
\end{eqnarray*}
where $ \mathbf{a}' = (k_9 \ell_8 \eta_2 \eta_5'^2, k_8 \ell_9 \eta_3 \eta_6'^2)$. We now aim to bound the contribution of the error term $\mathcal{E}$ by first estimating the quantity $\mathcal{N}$ and then by summing $\mathcal{N}$ over $\eta_1$, $\eta_2$, $\eta_3$ and $\eta_4$ and over all the possible values for the $Y_5$, $Y_6$ and $Y_7$. Note that the variables  $Y_5$, $Y_6$ and $Y_7$ satisfy the following inequalities
\begin{eqnarray}
\label{errorcondition1}
\eta_1^3 \eta_2^2 \eta_3^2 \eta_4^2 Y_5 Y_6 Y_7 & \leq & 3 B \textrm{,} \\
\label{errorcondition3}
\eta_4 Y_7^2 & \leq & 4 \eta_2 Y_5^2 \textrm{,} \\
\label{errorcondition4}
\eta_4 Y_7^2 & \leq & 4 \eta_3 Y_6^2 \textrm{.}
\end{eqnarray}
Applying lemma \ref{error} to sum over $\eta_5'$ and $\eta_6'$ and recalling that $q = k_{10} \eta_4 \eta_7^2$, we see that
\begin{eqnarray*}
\mathcal{N} & \ll & \log(B)^6 \sum_{Y_7 < \eta_7 \leq 2 Y_7} \sum_{k_8, k_9 \leq \mathcal{Z}^3} \sum_{k_{10} | \eta_1\eta_4} \sum_{\ell_8, \ell_9 | \eta_1\eta_2\eta_3}
\left( \frac{Y_5Y_6}{k_8k_9} + k_{10} \eta_4 \eta_7^2 \right) \tau(q)^2 E_1(q) \\
& \ll & \mathcal{Z}^7  \sum_{Y_7 < \eta_7 \leq 2 Y_7} \tau(\eta_1\eta_4) \tau(\eta_1\eta_2\eta_3)^2 \tau(\eta_1 \eta_4^2 \eta_7^2)^2
\left( Y_5Y_6 + \eta_1 \eta_4^2 \eta_7^2 \right) \\
& \ll & \mathcal{Z}^{12} \left( Y_5Y_6 Y_7 + \eta_1 \eta_4^2 Y_7^3 \right) \textrm{.}
\end{eqnarray*}
Using the two conditions \eqref{errorcondition3} and \eqref{errorcondition4}, we finally obtain
\begin{eqnarray*}
\mathcal{N} & \ll & \mathcal{Z}^{12} \eta_1 \eta_2^{1/2} \eta_3^{1/2} \eta_4 Y_5 Y_6 Y_7 \textrm{.}
\end{eqnarray*}
We now aim to sum this quantity over all the possible values for $Y_5$, $Y_6$ and $Y_7$. Let us start by summing over $Y_7$ using the condition \eqref{errorcondition1} and then over $\eta_1$ using the condition \eqref{log}, we obtain
\begin{eqnarray*}
\sum_{Y_i} \mathcal{N} & \ll & \mathcal{Z}^{12} \sum_{\eta_1, \eta_2, \eta_3, \eta_4, Y_5, Y_6, Y_7}  \eta_1 \eta_2^{1/2} \eta_3^{1/2} \eta_4 Y_5 Y_6 Y_7 \\
& \ll & B \mathcal{Z}^{13} \sum_{\eta_1, \eta_2, \eta_3, \eta_4} \frac1{\eta_1^2 \eta_2^{3/2} \eta_3^{3/2} \eta_4} \\
& \ll & B \mathcal{Z}^{-2} \sum_{\eta_2, \eta_3, \eta_4} \frac{\sq(\eta_2 \eta_3 \eta_4)}{\eta_2^{3/2} \eta_3^{3/2} \eta_4} \\
& \ll & B \mathcal{Z}^{-1} \textrm{,}
\end{eqnarray*}
which is satisfactory. In addition, the overall contributions of the three terms of the error term $\mathcal{E}'$ are easily seen to be bounded by
$B \log(B)^5 \log(\log(B))^{7/3}$, which is also satisfactory. Therefore, the main term of $N(\eta_1, \boldsymbol{\eta},B)$ is equal to
\begin{eqnarray*}
& & \sum_{\substack{k_8 | \eta_6, k_8 \leq \mathcal{Z}^3 \\ \gcd(k_8, \eta_1\eta_2\eta_3\eta_4\eta_7) = 1}} \frac{\mu(k_8)}{k_8} \sum_{\substack{k_9 | \eta_5, k_9 \leq \mathcal{Z}^3 \\ \gcd(k_9, \eta_1\eta_2\eta_3\eta_4\eta_7) = 1}} \frac{\mu(k_9)}{k_9} \sum_{\substack{k_{10} | \eta_1\eta_4 \\ \gcd(k_{10}, k_8 k_9 \eta_2 \eta_3 \eta_5 \eta_6) = 1}} \frac{\mu(k_{10})}{k_{10}} \\
& & \sum_{\substack{\ell_8, \ell_9 | \eta_1\eta_2\eta_3 \\ \gcd(\ell_8\ell_9,k_{10}\eta_4\eta_7) = 1}} \frac{\mu(\ell_8)}{\ell_8} \frac{\mu(\ell_9)}{\ell_9}
\varphi^{\ast}( k_{10} \eta_4 \eta_7) \frac{B^{2/3}}{\boldsymbol{\eta}^{(1/3,1/3,1/3,2/3,2/3,2/3)}} g_2 \left( \frac{\eta_1}{Z_1} \right) \textrm{.}
\end{eqnarray*}
Using the bound of lemma \ref{bounds} for $g_2$, we see that this quantity is
\begin{eqnarray*}
& \ll & \sum_{\substack{k_8 | \eta_6, k_9 | \eta_5 \\ k_8, k_9 \leq \mathcal{Z}^3}} \frac1{k_8} \frac1{k_9} \sigma_{-1}(\eta_1\eta_4) \sigma_{-1}(\eta_1\eta_2\eta_3)^2 \frac{B^{2/3}}{\boldsymbol{\eta}^{(1/3,1/3,1/3,2/3,2/3,2/3)}} \textrm{.}
\end{eqnarray*}
As a result, we see that if we remove the conditions $k_8, k_9 \leq \mathcal{Z}^3$ from the sums over $k_8$ and $k_9$, we create an error term whose overall contribution is for instance seen to be bounded by $B \mathcal{Z}^{-1}$. Thus, we have proved that we can write
\begin{eqnarray*}
N(\eta_1, \boldsymbol{\eta},B) & = & M(\eta_1, \boldsymbol{\eta},B) + R(\eta_1, \boldsymbol{\eta},B) \textrm{,}
\end{eqnarray*}
where
\begin{eqnarray*}
\sum_{\eta_1, \boldsymbol{\eta}} R(\eta_1, \boldsymbol{\eta},B) & \ll & B \log(B)^5\log(\log(B))^{7/3} \textrm{,}
\end{eqnarray*}
and
\begin{eqnarray*}
M(\eta_1, \boldsymbol{\eta},B) & = & \frac{B^{2/3}}{\boldsymbol{\eta}^{(1/3,1/3,1/3,2/3,2/3,2/3)}}
g_2 \left( \frac{\eta_1}{Z_1} \right) \theta (\eta_1, \boldsymbol{\eta}) \textrm{,}
\end{eqnarray*}
where
\begin{eqnarray*}
\theta(\eta_1, \boldsymbol{\eta}) & = & \sum_{\substack{k_8 | \eta_6 \\ \gcd(k_8, \eta_1\eta_2\eta_3\eta_4\eta_7) = 1}} \frac{\mu(k_8)}{k_8} \sum_{\substack{k_9 | \eta_5 \\ \gcd(k_9, \eta_1\eta_2\eta_3\eta_4\eta_7) = 1}} \frac{\mu(k_9)}{k_9} \\
& & \sum_{\substack{k_{10} | \eta_1\eta_4 \\ \gcd(k_{10}, k_8 k_9 \eta_2 \eta_3 \eta_5 \eta_6) = 1}} \frac{\mu(k_{10})}{k_{10}} \sum_{\substack{\ell_8, \ell_9 | \eta_1\eta_2\eta_3 \\ \gcd(\ell_8\ell_9,k_{10}\eta_4\eta_7) = 1}} \frac{\mu(\ell_8)}{\ell_8} \frac{\mu(\ell_9)}{\ell_9}
\varphi^{\ast}( k_{10} \eta_4 \eta_7) \\
& = &  \frac{\varphi^{\ast}(\eta_3\eta_6)}{\varphi^{\ast}(\eta_3)}  \frac{\varphi^{\ast}(\eta_2\eta_5)}{\varphi^{\ast}(\eta_2)} \varphi^{\ast}(\eta_1\eta_2\eta_3\eta_4\eta_7)^2
\sum_{\substack{k_{10} | \eta_1 \eta_4 \\ \gcd(k_{10},\eta_2 \eta_3 \eta_5 \eta_6)=1}} \frac{\mu(k_{10})}{k_{10} \varphi^{\ast}( \eta_4 \eta_7 k_{10})} \textrm{.}
\end{eqnarray*}
It is easy to check that for $a, b, c \geq 1$, we have
\begin{eqnarray*}
\sum_{\substack{k|a \\ \gcd(k,c) = 1}} \frac{\mu(k)}{k\varphi^{\ast}(kb)} & = &
\frac{\varphi^{\ast}(\gcd(a,b))}{\varphi^{\ast}(b)\varphi^{\ast}(\gcd(a,b,c))}
\prod_{\substack{p|a \\ p \nmid bc}} \left( 1 - \frac1{p-1} \right) \textrm{.}
\end{eqnarray*}
Using this equality and the remaining coprimality conditions \eqref{gcd4}, \eqref{gcd5} and \eqref{gcd6} and recalling the definition \eqref{checkmark} of $\psi$, we see that we can write
\begin{eqnarray*}
\theta(\eta_1, \boldsymbol{\eta}) & = & \theta_1(\eta_1, \boldsymbol{\eta}) \theta_2(\boldsymbol{\eta}) \textrm{,}
\end{eqnarray*}
where
\begin{eqnarray}
\label{theta1}
\theta_1(\eta_1, \boldsymbol{\eta}) & = & \psi_{\eta_2\eta_3\eta_4}(\eta_1) \textrm{,}
\end{eqnarray}
and
\begin{eqnarray}
\label{theta2}
\theta_2(\boldsymbol{\eta}) & = & \varphi^{\ast}(\eta_2\eta_3\eta_4)  \varphi^{\ast}(\eta_2\eta_3\eta_4\eta_5\eta_6\eta_7) \textrm{.}
\end{eqnarray}

\subsection{Summation over $\eta_1$}

We now need to sum the main term of $N(\eta_1,\boldsymbol{\eta},B)$ over $\eta_1 \in  \mathbb{Z}_{> 0}$ where $\eta_1$ is subject to the conditions \eqref{log} and \eqref{loglog} (the condition \eqref{height new} is implied by the definition of $g_2$) and to the coprimality condition \eqref{gcd4}. We start by proving that we can remove the restrictions that $\eta_1$ satisfies the conditions \eqref{log} and \eqref{loglog}. Indeed, let us first assume that we have the condition
\begin{eqnarray}
\label{opposite}
\eta_1 \sq(\eta_2 \eta_3 \eta_4) & < & B^{15/\log(\log(B))} \textrm{.}
\end{eqnarray}
The bound of lemma \ref{bounds} for $g_2$ implies that the main term $M(\eta_1, \boldsymbol{\eta},B)$ of $N(\eta_1,\boldsymbol{\eta},B)$ satisfies
\begin{eqnarray*}
M(\eta_1, \boldsymbol{\eta},B) & \ll & \frac{B^{2/3}}{\boldsymbol{\eta}^{(1/3,1/3,1/3,2/3,2/3,2/3)}} \textrm{.}
\end{eqnarray*}
Let us now sum this quantity over $\eta_7$ using the condition \eqref{height new} and then over $\eta_1$ using the condition \eqref{opposite}, we obtain
\begin{eqnarray*}
\sum_{\eta_1, \boldsymbol{\eta}} M(\eta_1, \boldsymbol{\eta},B) & \ll &
\sum_{\eta_1, \eta_2, \eta_3, \eta_4, \eta_5, \eta_6} \frac{B}{\eta_1\boldsymbol{\eta}^{(1,1,1,1,1,0)}} \\
& \ll & \sum_{\eta_2, \eta_3, \eta_4, \eta_5, \eta_6} \frac{B \log(B)}{\boldsymbol{\eta}^{(1,1,1,1,1,0)}\log(\log(B))} \\
& \ll & \frac{B \log(B)^6}{\log(\log(B))} \textrm{.}
\end{eqnarray*}
This error term is satisfactory. Let us now assume that we have the condition
\begin{eqnarray*}
\eta_1^3 \eta_2^2 \eta_3^2 \eta_4^2 \eta_5 \eta_6 \eta_7 & < & \frac{B}{\log(\log(B))} \textrm{.}
\end{eqnarray*}
Let us sum over $\eta_1$ using this condition, we get
\begin{eqnarray*}
\sum_{\eta_1, \boldsymbol{\eta}} M(\eta_1, \boldsymbol{\eta},B) & \ll &
\sum_{\boldsymbol{\eta}} \frac{B}{\boldsymbol{\eta}^{(1,1,1,1,1,1)} \log(\log(B))^{1/3}} \\
& \ll & \frac{B \log(B)^6}{\log(\log(B))^{1/3}} \textrm{.}
\end{eqnarray*}
This error term is also satisfactory. We can thus remove the restrictions that $\eta_1$ satisfies the conditions \eqref{log} and \eqref{loglog} and we proceed to sum over $\eta_1$.  Recall the definition \eqref{V} of $\mathcal{V}$. For fixed  $\boldsymbol{\eta} \in \mathcal{V}$ satisfying the coprimality conditions \eqref{gcd5} and \eqref{gcd6}, let $N(\boldsymbol{\eta},B)$ be the sum of the main term of
$N(\eta_1,\boldsymbol{\eta},B)$ over $\eta_1$, where $\eta_1$ is subject to the coprimality condition \eqref{gcd4}. Recall the definition \eqref{Upsilon} of $\Upsilon$. We now prove the following lemma.

\begin{lemma}
\label{summation 7}
We have the estimate
\begin{eqnarray*}
N(\boldsymbol{\eta},B) & = & \Upsilon \frac{\omega_{\infty}}{3} \frac{B}{\boldsymbol{\eta}^{(1,1,1,1,1,1)}} \Theta(\boldsymbol{\eta}) + R(\boldsymbol{\eta},B) \textrm{,}
\end{eqnarray*}
where $\Theta(\boldsymbol{\eta})$ is a certain arithmetic function defined in \eqref{Theta} and where
\begin{eqnarray*}
\sum_{\boldsymbol{\eta}} R(\boldsymbol{\eta},B) & \ll & B \log(B)^5 \textrm{.}
\end{eqnarray*}
\end{lemma}

\begin{proof}
Let us use lemma \ref{arithmetic preliminary 1} to sum over $\eta_1$. For any fixed $0 < \gamma \leq 1$, we obtain
\begin{eqnarray*}
N(\boldsymbol{\eta},B) & = & \Upsilon \frac{B}{\boldsymbol{\eta}^{(1,1,1,1,1,1)}} \Theta(\boldsymbol{\eta}) \int_{t > 0} g_2(t) \D t \\
\notag
& & + O \left( \frac{B^{2/3}}{\boldsymbol{\eta}^{(1/3,1/3,1/3,2/3,2/3,2/3)}} Z_1^{\gamma} \sigma_{-\gamma/2}(\eta_2 \eta_3 \eta_4 \eta_5 \eta_6 \eta_7)
\sup_{t > 0} g_2(t) \right) \textrm{,}
\end{eqnarray*}
where
\begin{eqnarray}
\label{Theta}
\ \ \ \ \ \Theta(\boldsymbol{\eta}) & = & \varphi^{\ast}(\eta_2\eta_3\eta_4) \varphi^{\ast}(\eta_2\eta_3\eta_4\eta_5\eta_6\eta_7) \varphi^{\ast}(\eta_5\eta_6\eta_7)
\varphi^{\curlyvee}(\eta_2\eta_3\eta_4\eta_5\eta_6\eta_7) \textrm{.}
\end{eqnarray}
Let us set $\gamma = 1/2$. Using the bound of lemma \ref{bounds} for $g_2$, we deduce that the overall contribution of this error term is
\begin{eqnarray*}
\sum_{\boldsymbol{\eta}} \frac{B^{5/6}}{\boldsymbol{\eta}^{(2/3,2/3,2/3,5/6,5/6,5/6)}} \sigma_{-1/4}(\eta_2 \eta_3 \eta_4 \eta_5 \eta_6 \eta_7) & \ll & B \log(B)^5 \textrm{,}
\end{eqnarray*}
where we have summed over $\boldsymbol{\eta}$ using the condition $Z_1 \geq 3^{-1/3}$. Recalling the definition of $g_2$ and the equality \eqref{omega}, we see that
\begin{eqnarray*}
\int_{t > 0} g_2(t) \D t & = & \frac{\omega_{\infty}}{3} \textrm{,}
\end{eqnarray*}
which completes the proof.
\end{proof}

\subsection{Conclusion}

It remains to sum the main term of $N(\boldsymbol{\eta},B)$ over the $\boldsymbol{\eta} \in \mathcal{V}$ satisfying the coprimality conditions \eqref{gcd5} and \eqref{gcd6}. It is easy to see that replacing $\mathcal{V}$ by the region
\begin{eqnarray*}
\mathcal{V}' & = & \left\{ \boldsymbol{\eta} \in \mathbb{Z}_{> 0}^6,
\begin{array}{l}
Y \geq q_8 Z_1^2, Y \geq q_9 Z_1^2 \\
Z_1 \geq 1, q_8 \geq q_{10}, q_9 \geq q_{10}
\end{array}
\right\} \textrm{,}
\end{eqnarray*}
produces an error term whose overall contribution is $\ll B \log(B)^5 \log(\log(\log(B)))$. Let us redefine the arithmetic function $\Theta$ as being equal to zero if the remaining coprimality conditions \eqref{gcd5} and \eqref{gcd6} are not satisfied. Recalling lemma \ref{N(B)}, we see that we have proved the following lemma.

\begin{lemma}
\label{final lemma}
We have the estimate
\begin{eqnarray*}
N_{U,H}(B) & = & \Upsilon \omega_{\infty} B \sum_{\boldsymbol{\eta} \in \mathcal{V}'} \frac{\Theta(\boldsymbol{\eta})}{\boldsymbol{\eta}^{(1,1,1,1,1,1)}}
+ O \left( \frac{B \log(B)^6}{\log(\log(B))^{1/6}} \right) \textrm{.}
\end{eqnarray*}
\end{lemma}

The end of the paper is dedicated to the completion of the proof of theorem \ref{Manin}. Let us introduce the generalized Möbius function  $\boldsymbol{\mu}$ defined for $(n_1, \dots, n_6) \in \mathbb{Z}_{>0}^6$ by
$\boldsymbol{\mu}(n_1, \dots, n_6) = \mu(n_1) \cdots \mu(n_6)$. We set $\mathbf{k} = (k_2,k_3,k_4,k_5,k_6,k_7)$ and we define for $s \in \mathbb{C}$ such that $\Re(s) > 1$,
\begin{eqnarray*}
F(s) & = & \sum_{\boldsymbol{\eta} \in \mathbb{Z}_{>0}^6}
\frac{\left|(\Theta \ast \boldsymbol{\mu})(\boldsymbol{\eta})\right|}{\eta_2^s \eta_3^s \eta_4^s \eta_5^s \eta_6^s \eta_7^s} \\
& = & \prod_p \left( \sum_{\mathbf{k} \in \mathbb{Z}_{\geq 0}^6}
\frac{\left|(\Theta \ast \boldsymbol{\mu}) \left( p^{k_2},p^{k_3},p^{k_4},p^{k_5}, p^{k_6}, p^{k_7} \right)\right|}
{p^{k_2 s} p^{k_3 s} p^{k_4 s} p^{k_5 s} p^{k_6 s} p^{k_7 s}} \right) \textrm{.}
\end{eqnarray*}
It is easy to check that if $\mathbf{k} \notin \{0,1\}^6$ then
$(\Theta \ast \boldsymbol{\mu}) \left( p^{k_2},p^{k_3},p^{k_4},p^{k_5},p^{k_6}, p^{k_7} \right) = 0$ and if exactly one of the $k_i$ is equal to $1$, then
$(\Theta \ast \boldsymbol{\mu}) \left( p^{k_2},p^{k_3},p^{k_4},p^{k_5}, p^{k_6}, p^{k_7} \right) \ll 1/p$, so the local factors $F_p$ of $F$ satisfy
\begin{eqnarray*}
F_p(s) & = & 1 + O \left( \frac1{p^{ \min \left( \Re(s)+1, 2 \Re(s) \right)}} \right) \textrm{.}
\end{eqnarray*}
This proves that the function $F$ converges in the half-plane $\Re(s) > 1/2$, which implies that $\Theta$ satisfies the assumption of \cite[Lemma $8$]{MR2853047}. The application of this lemma provides
\begin{eqnarray}
\label{sum1}
\ \ \ \ \ \ \ \sum_{\boldsymbol{\eta} \in \mathcal{V}'} \frac{\Theta(\boldsymbol{\eta})}{\boldsymbol{\eta}^{(1,1,1,1,1,1)}} & = & \alpha \left( \sum_{\boldsymbol{\eta} \in \mathbb{Z}_{>0}^6}
\frac{(\Theta \ast \boldsymbol{\mu})(\boldsymbol{\eta})}{\boldsymbol{\eta}^{(1,1,1,1,1,1)}} \right) \log(B)^6 + O \left( \log(B)^5 \right) \textrm{,}
\end{eqnarray}
where $\alpha$ is the volume of the polytope defined in $\mathbb{R}^6$ by $t_2,t_3,t_4,t_5,t_6,t_7 \geq 0$ and
\begin{eqnarray*}
2 t_2 - t_3 - t_4 + 4 t_5 - 2 t_6 - 2 t_7 & \leq & 1 \textrm{,} \\
- t_2 + 2 t_3 - t_4 -2 t_5 + 4 t_6 -2 t_7 & \leq & 1 \textrm{,} \\
2 t_2 + 2 t_3 +2 t_4 + t_5 + t_6 + t_7 & \leq & 1 \textrm{,} \\
- t_2 + t_4 - 2 t_5 + 2 t_7 & \leq & 0 \textrm{,} \\
- t_3 + t_4 - 2 t_6 + 2 t_7 & \leq & 0 \textrm{.}
\end{eqnarray*}
It is easy to achieve the computation of $\alpha$ using Franz's additional \textit{Maple} package Convex \cite{Convex}. We find
$\alpha = 1/23040$, that is to say
\begin{eqnarray}
\label{alpha}
\alpha & = & \alpha(\widetilde{V}) \textrm{.}
\end{eqnarray}
Furthermore, we have
\begin{eqnarray*}
\sum_{\boldsymbol{\eta} \in \mathbb{Z}_{>0}^6} \frac{(\Theta \ast \boldsymbol{\mu}) (\boldsymbol{\eta})}{\boldsymbol{\eta}^{(1,1,1,1,1,1)}} & = &
\prod_p \left( \sum_{\mathbf{k} \in \mathbb{Z}_{\geq 0}^6}
\frac{(\Theta \ast \boldsymbol{\mu}) \left( p^{k_2},p^{k_3},p^{k_4},p^{k_5},p^{k_6},p^{k_7} \right)}{p^{k_2}p^{k_3}p^{k_4}p^{k_5}p^{k_6}p^{k_7}} \right) \\
& = & \prod_p \left( 1 - \frac1{p} \right)^6 \left( \sum_{\mathbf{k} \in \mathbb{Z}_{\geq 0}^6}
\frac{\Theta \left( p^{k_2},p^{k_3},p^{k_4},p^{k_5},p^{k_6},p^{k_7} \right)}{p^{k_2}p^{k_3}p^{k_4}p^{k_5}p^{k_6}p^{k_7}} \right) \textrm{.}
\end{eqnarray*}
The calculation of these local factors is straightforward and we find
\begin{eqnarray*}
\sum_{\mathbf{k} \in \mathbb{Z}_{\geq 0}^6}
\frac{\Theta \left( p^{k_2},p^{k_3},p^{k_4},p^{k_5},p^{k_6},p^{k_7} \right)}{p^{k_2}p^{k_3}p^{k_4}p^{k_5}p^{k_6}p^{k_7}} & = &
\varphi^{\curlyvee}(p) \left( 1 - \frac1{p} \right) \left( 1 + \frac{7}{p} + \frac1{p^2} \right) \textrm{.}
\end{eqnarray*}
We finally obtain
\begin{eqnarray}
\label{sum2}
\sum_{\boldsymbol{\eta} \in \mathbb{Z}_{>0}^6} \frac{(\Theta \ast \boldsymbol{\mu}) (\boldsymbol{\eta})}{\boldsymbol{\eta}^{(1,1,1,1,1,1)}} & = &
\Upsilon^{-1} \prod_p \left( 1 - \frac1{p} \right)^7 \omega_p \textrm{.}
\end{eqnarray}
Putting together the equalities \eqref{sum1}, \eqref{alpha}, \eqref{sum2} and lemma \ref{final lemma} completes the proof of theorem \ref{Manin}.

\bibliographystyle{amsalpha}
\bibliography{biblio}

\end{document}